\documentclass[12pt]{amsart}

\usepackage{amscd}
\usepackage{amssymb}

\usepackage[ignorehead,top=2in, bottom=2in, left=1in, right=1in]{geometry}

\usepackage{cleveref}
\crefname{section}{§}{§§}

\usepackage{graphicx}
\usepackage[font={small,it},labelfont=bf]{caption}

\usepackage{microtype}

\usepackage{setspace}

\newcommand{\cA}{\mathcal A}
\newcommand{\cB}{\mathcal B}
\newcommand{\cC}{\mathcal C}

\newcommand{\cE}{\mathcal E}

\newcommand{\cG}{\mathcal G}

\newcommand{\cL}{\mathcal L}
\newcommand{\cM}{\mathcal M}

\newcommand{\cR}{\mathcal R}

\newcommand{\bZ}{\mathbb Z}
\newcommand{\bQ}{\mathbb Q}
\newcommand{\bC}{\mathbb C}
\newcommand{\bR}{\mathbb R}

\newcommand{\vv}{\, | \,}

\providecommand{\norm}[1]{\lVert #1\rVert}

\DeclareMathOperator{\Fix}{Fix}
\DeclareMathOperator{\Sign}{Sign}
\DeclareMathOperator{\sign}{sign}

\DeclareMathOperator{\Aut}{Aut}
\DeclareMathOperator{\End}{End}
\DeclareMathOperator{\ad}{ad}

\DeclareMathOperator{\Hol}{Hol}

\DeclareMathOperator{\Id}{Id}
\DeclareMathOperator{\Ind}{Ind}
\DeclareMathOperator{\Diff}{Diff}

\DeclareMathOperator{\ch}{ch}
\DeclareMathOperator{\csch}{csch}
\DeclareMathOperator{\Tr}{Tr}

\newtheorem{thm}{Theorem}[section]
\newtheorem{lem}[thm]{Lemma}
\newtheorem{prop}[thm]{Proposition}

\newtheorem*{cor*}{Corollary}
\newtheorem*{thmA}{Theorem A}
\newtheorem*{thmB}{Theorem B}

\newtheorem*{corC}{Corollary C}

\theoremstyle{definition}

\newtheorem{ex}[thm]{Example}
\newtheorem*{ack}{Acknowledgments}

\newtheorem{remark}{Remark}[section]

\numberwithin{equation}{section}

\vbadness=\maxdimen

\begin{document}

\onehalfspace

\title[Cyclic Actions and the Poincar\'e Homology $3$-Sphere]{Extending Smooth Cyclic Group Actions On The Poincar\'e Homology Sphere}
\author{Nima Anvari}
\date{\today}
\email{anvarin@mpim-bonn.mpg.de}

\begin{abstract}
Let $X_0$ denote a compact, simply-connected smooth $4$-manifold with boundary the Poincar\'e homology $3$-sphere $\Sigma(2,3,5)$ and with even negative definite intersection form $Q_{X_0}=E_8$. We show that free $\bZ/p$ actions on $\Sigma(2,3,5)$ do not extend to smooth actions on $X_0$ with isolated fixed points for any prime $p>7$. The approach is to study the equivariant version of the Yang-Mills instanton-one moduli space for $4$-manifolds with cylindrical ends. As an application we show that for $p>7$ a smooth $\bZ/p$ action on $\#^8 S^2 \times S^2$ with isolated fixed points does not split along a free action on $\Sigma(2,3,5)$. The results hold for $p=7$ if the action is homologically trivial.
\end{abstract}

\maketitle

\section{Introduction}

Development of gauge theoretic techniques have led to substantial new insights in low dimensional topology and the area of smooth finite group actions on $4$-manifolds is no exception. One of the earliest applications was by Furuta \cite{F89} on the non-existence of one fixed point actions of a finite group on homotopy four spheres. The idea was to study the instanton-one Yang-Mills moduli space of $SU(2)$ anti-self dual connections together with its induced group action. Later Braam and Matic developed this further leading to new proof of the Smith conjecture in dimension four \cite{BM93}. A new complication which arises in the equivariant setting is the existence of non-trivial obstructions to transversality. Hambleton-Lee in \cite{HL92} formulated a suitable perturbation theory based on Bierstone's notion of equivariant general position \cite{Bie77} - this gives the equivariant moduli spaces the structure of a Whitney stratified space and by appealing further to the internal structure of the moduli spaces they were able to show that homologically-trivial smooth actions on closed, simply-connected odd definite $4$-manifolds resemble equivariant connected sums of linear algebraic actions on complex projective space (\cite{HL95}, see also \cite{HT04} for the case of non-trivial action on homology). A theorem of Edmonds \cite{E87} shows that for primes $p>3$ topological $4$-manifolds always admit homologically trivial locally linear $\bZ/p$ actions with fixed set consisting of isolated fixed points; when combined with the constraints given by gauge theory this led to the existence of non-smoothable locally linear $\bZ/p$ actions (\cite{KL93} using the theory of pseudofree orbifolds as in \cite{FS85}, see also \cite{KL88} and \cite{HL95}). Indefinite $4$-manifolds presented new difficulties but since then results have been obtained on locally linear but non-smoothable actions using ideas from Seiberg-Witten gauge theory (see \cite{NL2008}, \cite{CK2008} and \cite{kiyono2011}). 

Before the Seiberg-Witten revolution however, there was much effort to lay the foundations for the theory of instanton moduli spaces on $4$-manifolds with cylindrical ends (\cite{T87},\cite{taubes1993},\cite{MMR}) with various applications that soon followed (see for instance \cite{GM93} and \cite{FS94}). One motivation for this paper was the work of Fintushel and Stern in \cite{FS91} to splitting $4$-manifolds along certain Seifert fibered homology $3$-spheres. We would like to revisit these results from an equivariant point of view and use them to study smooth equivariant splittings of $4$-manifolds along integral homology spheres. Recall by Freedman and Taylor \cite{FT77} whenever we decompose the intersection form into orthogonal components $Q_X=Q_{X_0}\oplus Q_{X_1}$ there exists a diffeomorphism $X=X_0 \cup_{\Sigma} X_1$ splitting $X$ along an integral homology $3$-sphere $\Sigma$. In this paper we will ask if this can be done equivariantly for $\bZ/p$ actions on $X=\#^8 S^2 \times S^2$ inducing a free action along the Poincar\'e homology $3$-sphere. 

This homology sphere is obtained by $(-1)$-surgery on the left handed trefoil knot but more conveniently thought of as a link in the complex surface singularity $\Sigma(2,3,5)=\{(z_1,z_2,z_3)\in \bC^3 \vv z_1^{2}+z_2^{3}+z_3^{5}=0\}\cap S^5$. If $X_0$ is a definite, smooth, compact and simply-connected $4$-manifold with boundary $\Sigma(2,3,5)$ then the intersection form is $Q_{X_0}=n(-1)\oplus E_8$ with $n\geq 0$ \cite{Froyshov96}, we will only consider even-definite intersection forms.
\begin{thmA}
Let $X_0$ denote a smooth, simply-connected $4$-manifold with boundary $\Sigma(2,3,5)$ and negative definite intersection form $Q_{X_0}=E_8$. For any prime $p>5$, if a free $\bZ/p$-action on $\Sigma(2,3,5)$ extends to a smooth action on $X_0$, then the rotation data of the isolated fixed points are $(a,b)$ such that $a+b \equiv \pm 1 \pmod p$ or $a+b \equiv \pm 7 \pmod p$.
\end{thmA}
\begin{remark}
Let $G$ denote the finite cyclic group $\bZ/p$, then the action on $X_0$ gives rise to an integral representation on $H_2(X_0;\bZ)$ and a decomposition (see \cite[p.508]{CR62}, \cite[p.111]{edmonds1989aspects}) $H_2(X_0;\bZ)=\bZ[\bZ/p]^r\oplus \bZ^t \oplus \bZ[\zeta_p]^c$ as $\bZ[G]$-modules with $r,s,c \geq 0$ and $b_2(X_0)=rp+t+(p-1)c$. When $p > b_2(X_0)+1$ we must have $r,c=0$ and $t=b_2(X_0)$, as a result the action is homologically trivial for $p>7$. When $p=7$ the action may not be homologically trivial (see \cite[Example 3.10, p. 168]{quebbemann1981klassifikation}) and $H_2(X_0;\bZ)=\bZ[G]\oplus \bZ$ (since by \cite[Proposition 2.4 i)]{edmonds1989aspects} $c$ must be even, see also the algebraic result in \cite[Proposition 10 c)]{HR78}). In the case of homologically trivial actions, the fixed set consists of isolated points and $2$-spheres and the Lefschetz fixed point formula gives the Euler characteristic $\chi(\Fix(X_0,G))=9$.
\end{remark}
The rotation data are the tangential representations over a fixed point $\bC^2(a,b)$ where the action is given by $t\cdot(z_1,z_2)=(t^a z_1, t^b z_2)$ for non-zero integers $a$ and $b$ well-defined modulo $p$ and where $\bZ/p=\langle t \rangle$. The restriction on the primes is to ensure a free action on $\Sigma(2,3,5)$. When $p$ is relatively prime to $a,b,c$, the free $\bZ/p$ action on $\Sigma(a,b,c)$ is part of the circle action $t \times (x,y,z)=(t^{bc}x,t^{ac}y,t^{ab}z)$ of the Seifert fibered structure on the Brieskorn spheres $\Sigma(a,b,c)$. It can be shown (see \cite{LS92}) that up to equivalence, this is the only free action on $\Sigma(a,b,c)$, referred to as the \emph{standard action}. The necessary conditions for a smooth extension from Theorem A can be checked against the $G$-signature theorem for manifolds with boundary where the correction term is an Atiyah-Patodi-Singer type invariant(see \cite{APS2}) and this leads to the following rigidity result 
\begin{thmB}
Let $X_0$ denote a smooth, simply-connected $4$-manifold with boundary $\Sigma(2,3,5)$ and negative definite intersection form $Q_{X_0}=E_8$. A free $\bZ/p$ action on $\Sigma(2,3,5)$ does not extend to a smooth action on $X_0$ with fixed set consisting of isolated fixed points for any prime $p>7$. The same is true for $p=7$ if the extended action is homologically trivial.
\end{thmB}
As an immediate consequence we have the following corollary regarding equivariant embedding of the Poincar\'e homology sphere in $\#^8 S^2 \times S^2$.
\begin{corC}
For $X=\#^8 S^2 \times S^2$ with a smooth $\bZ/p$ action with isolated fixed points, does not contain an equivariant embedding of $\Sigma(2,3,5)$ with a free action, for any prime $p>7$ and for $p=7$ if the action is homologically trivial on $X$.
\end{corC}
\begin{remark}
If $\Sigma(2,3,5)$ embeds smoothly in $\#^8 S^2 \times S^2$ then it separates $X$ into two smooth, spin $4$-manifolds with boundary each with even intersection form. By van der Blij's lemma each side must have signature divisible by $8$ and since $b_2(X)=16$ each side must have definite intersection form. The additivity of the signature shows that they must have opposite sign. So any embedding of $\Sigma(2,3,5)$ into  $\#^8 S^2 \times S^2$ splits $X$ as $X_0\cup_{\Sigma(2,3,5)} X_1$ with $Q_{X_0}=E_8=-Q_{X_1}$. 
\end{remark}
Note from the construction in Edmonds \cite[p.161]{E87}, it is possible to do equivariant surgery on a framed link $K \subset S^3$ to obtain a compact, simply-connected four manifold $X_0$, whose intersection form is $E_8$ and admits a smooth, homologically trivial $\bZ/p$-action with fixed set consisting of isolated points. Since the intersection form is unimodular, the boundary consists of a integral homology three sphere $\Sigma$ with a free $\bZ/p$-action. By reversing orientation and gluing, we see that it is possible to smoothly and equivariantly split $\#^8 S^2 \times S^2=E_8 \cup_{\Sigma} -E_8$ along \emph{some} homology sphere $\Sigma$ in this way. The content of Theorem B is that this splitting cannot occur along the Poincar\'e homology $3$-sphere. Moreover, if we relax the condition that the action have only isolated fixed points, then equivariant plumbing on the $E_8$ diagram produces a smooth, homologically trivial $\bZ/p$-equivariant splitting of $\#^8 S^2 \times S^2=E_8 \cup_{\Sigma(2,3,5)} -E_8$ with fixed set consisting of two fixed $2$-spheres representing homology classes with self-intersection $-2$ and $14$ isolated fixed points. 

\begin{ack}
This work was part of the author's PhD thesis under the supervision of Prof.~Ian Hambleton at McMaster University. The author is grateful for his great patience, expertise and support while working on this project. Partial support was provided by Ontario Graduate Scholarship (OGS 2010-2012). The author would also like to thank Prof.~Nikolai Saveliev for many helpful discussions and to Prof.~Ron Fintushel for pointing out the application of Theorem A to splitting $\#^8 S^2 \times S^2$; our original motivation was to consider equivariant embedding of $\Sigma(2,7,13)$ in homotopy $K3$ surfaces. A portion of this paper was written at the Max Planck Institute for Mathematics in Bonn and the author wishes to thank them for their hospitality during this visit. 
\end{ack}

\section{Equivariant Moduli Spaces}

We begin by reviewing the cylindrical-end theory, for details the reader should consult \cite{MMR} and \cite{D02}. The general setup concerns an integral homology $3$-sphere $\Sigma$ with non-degenerate flat connections (i.e. the twisted cohomology $H^1(\Sigma, \ad \alpha)=0$ for all flat connections $\alpha$) and $X_0$ a compact, smooth, simply-connected $4$-manifold with negative definite intersection form with boundary $\partial X_0=\Sigma$, together with a smooth group action of $\pi=\bZ/p$ which is homologically-trivial on $X_0$ and free on $\Sigma$. Denote by $(X,g)$ the non-compact complete Riemannian manifold $X=X_0 \cup \End(X)$ where $\End(X)=\Sigma \times [0,\infty)$ and $g$ is a $\pi$-invariant metric which restricts to the product metric on the end: $g \vv_{\End(X)}=ds^2+g_{\Sigma}$; the $\pi$-action on $X_0$ is extended in the obvious way on the infinite half-cylinder. 

Consider a principal $SU(2)$ bundle $P$ over $X$ (necessarily trivial), by fixing a trivialization we get bundle maps $\widehat{t}: P \rightarrow P$ which cover the $\pi$-action on $X$. Let $\cG(\pi)=\{ \widehat{t}:P \rightarrow P\vv t \in \pi \}$, then there exists an exact sequence
\begin{equation}
1 \rightarrow \cG \rightarrow \cG(\pi) \rightarrow \pi \rightarrow 1
\end{equation}
where $\cG$ is the gauge group of $P$. The natural action of $\cG(\pi)$ on the space of connections $\cA(P)$ is given by pull-back, and well-defined modulo gauge. So the space of connections $\cB(P)=\cA/\cG$ up to gauge transformations inherits a $\cG(\pi)/\cG=\pi$-action. Recall the definition of the Yang-Mills energy functional acting on the space of connections $\mathcal{YM}:\mathcal{A}(P)\rightarrow [0,\infty]$ is defined by
\begin{equation}
\mathcal{YM}(A)=\dfrac{-1}{8\pi^2}\int_X Tr(F_A \wedge \ast F_A)=\dfrac{1}{8\pi^2}\int_X \vv F_A \vv^2 =\norm{F_A}_{L^2}^2
\end{equation}
where $F_A$ is $\ad(P)$-valued curvature $2$-form of the connection and $*$ is the Hodge star operator associated to the Riemannian metric. This functional is both conformally-invariant and gauge-invariant under the action of $\cG(P)$ and so descends to a well-defined functional on $\cB(P)$ depending only on the conformal class $[g]$. The critical points satisfy the Yang-Mills equation: $d_A^{\ast}F_A=0$. The Hodge $\ast$-star operator on $X$ extends to an involution on bundle valued $2$-forms giving rise to a splitting 
\begin{equation}
\Lambda^2(X)\otimes \ad P = (\Lambda^2_+ \otimes \ad P) \oplus (\Lambda^2_- \otimes \ad P)
\end{equation}
and the corresponding orthogonal $\pm 1$ - eigenspace decomposition:
\begin{equation}
\Omega^2(\ad P)=\Omega^2_{+}(\ad P)\oplus \Omega^2_{-}(\ad P)
\end{equation}
into self-dual and anti-self dual(ASD) $2$-forms. Anti-self dual connections $F_A^+=0$ automatically satisfy the Yang-Mills equations by the Bianchi identity. The $L^2$-finite moduli-spaces are anti-self dual connections modulo gauge with finite Yang-Mills action: 
\begin{equation}
\cM(X,g)=\{[A] \in \cB(P) \vv F_A^+=0, \norm{F_A}_{L^2}^2 < \infty \}.
\end{equation}
This space is $\pi$-invariant when the action is by isometries. It is a fundamental result that $g$-ASD connections with finite Yang-Mills energy are asymptotic to flat connections down the cylindrical-end (see \cite{T87},\cite{MMR},or \cite{D02}). Let $i_t: \Sigma \times \{ t\} \hookrightarrow \Sigma \times [0,\infty)$ denote the inclusion map, then this convergence result tells us that in this non-degenerate setting we have exponential decay to a flat connection and a well-defined limit $\lim_{t \rightarrow \infty}[i_t^{*}A\vv_{\End{X}}]$ in $\cR(\Sigma)$, the representation variety of flat connections modulo gauge. Since flat connections are $\pi$-invariant under the action of pull back, this defines a $\pi$-equivariant boundary map $\partial_{\infty}$ 
\begin{gather}
\partial_{\infty}:  \cM(X,g) \rightarrow \cR(\Sigma) \notag\\
[A] \mapsto \lim_{t \rightarrow \infty}[i_t^{*}A\vv_{\End{X}}]
\end{gather}
We can partition the $g$-ASD moduli space according to its limiting flat connection: 
\begin{equation}
\cM(X,g)=\bigsqcup_{\alpha \in \cR(Y)} \cM(X,\alpha).
\end{equation}
If $X$ has more than one end, then there exists a boundary map for each end. When $X$ is the cylinder $\Sigma \times \bR$ with product metric $g$ then there exist finite energy $g$-ASD connections on the cylinder if and only if there exists a gradient-flow line for the Chern-Simons functional connecting the flat connections on the ends \cite{Floer88}. The energy of a $g$-ASD connection $A$ is given by the second relative Chern number
\begin{equation}
\ell=c_2(A)=\dfrac{1}{8\pi^2}\int_X Tr(F_{A}^2)
\end{equation} 
and this value is congruent modulo $\bZ$ to the Chern-Simons invariant of the limiting flat connection $\alpha$. Thus for $g$-ASD connections with finite Yang-Mills action, the energy takes on a discrete set of values determined by the Chern-Simons invariant and we get a further decomposition according to energy value 
$\cM(X,\alpha)=\bigsqcup_{\ell \geq 0} \cM_{\ell}(X,\alpha)$ with $\ell \equiv CS(\alpha) \mod \bZ$.

Let $\delta$ be a positive constant and $\tau : X \rightarrow [0,\infty)$ denote a smoothing of the ``time" parameter on the end $\Sigma \times [0,\infty)$ to a function that is zero on $X_0$ and $\Omega^i_{k,\delta}(X,\ad P)$ to be the completion of compactly supported forms $\Omega^i(X,\ad P)$ by the weighted Sobolev norm $L^2_{k,\delta}$: 
\begin{equation}
\norm{a}^2_{k,\delta}=\int_X e^{\delta \tau}\{\vv d_{A_0}^k a \vv^2+\cdots+\vv d_{A_0}a\vv^2+\vv a\vv^2\}.
\end{equation}
Let $A_0$ denote a connection which extends the given flat connection $\alpha$ and has second relative Chern number $c_2(A_0)=\ell$. Denote by $\cA_{\ell}(\alpha)=\{A_0+a \vv a \in  \Omega^1_{2,\delta}(X,\ad P)\}$ and $\cG_{\ell}(\alpha)=\{u \in \cG\vv d_{A_0}u \in \Omega^1_{2,\delta}(X,\ad P)\}$ respectively, the space of connections that limit exponentially to $\alpha$ on the cylindrical-end and gauge transformations that stabilize $\alpha$ on the end. There is a slice theorem for the action of $\cG_{\ell}^{0}(\alpha)$(the analogue of the based gauged group) on $\cA_{\ell}(\alpha)$ given by the Coulomb gauge condition $d^{\ast}_{A,\delta}a=0$, where $d_{A,\delta}^{\ast}=e^{-\delta \tau}d_A^{\ast}e^{\delta \tau}$ is the adjoint operator with respect to the weighted norm. The framed moduli space $\cM_{\ell}^{0} \subset \cA(\alpha)/\cG_{\ell}^0$ admits a smooth $\Gamma_{\alpha}$ action (the stabilizer of $\alpha$) with the quotient being the genuine moduli space $\cM_{\ell}(X,\alpha)$ of exponentially decaying $g$-ASD connections. The results on exponential decay tell us that these function spaces and hence the moduli space $\cM_{\ell}(X,\alpha)$ capture all the finite-energy instantons that are asymptotic to $\alpha$ when $\delta$ is suitably chosen.

The infinitesimal deformations of an ASD connection $[A] \in \cM(X,g)$ is given by the first cohomology of the fundamental elliptic $\delta$-decay complex:
\begin{equation}
0 \rightarrow \Omega^0_{3,\delta}(X,\ad P) \xrightarrow{d_{A}} \Omega^1_{2,\delta}(X,\ad P) \xrightarrow{d_{A}^{+}} \Omega^2_{1,\delta,+}(X,\ad P) \rightarrow 0
\end{equation}
with associated elliptic operator 
\begin{equation}
D_{A,\delta}=d_{A,\delta}^{\ast}+d_A^+ : \Omega^1_{2,\delta}(X,\ad P) \rightarrow \Omega^0_{3,\delta}(X,\ad P) \oplus \Omega^2_{1,\delta,+}(X,\ad P)
\end{equation}
where over the $\End(X)$ this operator can be written as 
\begin{equation}
D_{A,\delta}\vv_{\End(X)}=\dfrac{\partial}{\partial t}+(L_{A(t)}-\delta)
\end{equation}
for a self-adjoint elliptic operator $L_{A(t)}$:
\begin{equation}
L_{A(t)}=
\begin{pmatrix}
0 & -d_{A(t)}^{\ast}\\
-d_{A(t)}& \ast_{3}d_{A(t)}
\end{pmatrix}
\end{equation}
with pure point real spectrum without accumulation points. The operator $D_{A,\delta}$ is Fredholm for $\delta$ not in the spectrum of $L_{\alpha}$, where $\alpha$ is the limiting flat connection of $A$. The formal dimension of the framed moduli space $\cM^0_{\ell}(X,\alpha)$ is given by $\Ind D_{A,\delta}=\Ind D_A +h^{0}_{\alpha}$ and can be computed from \cite{APS1}, see \cite[p. 138]{Nikolai}:
\begin{equation*}
\begin{split}
\Ind D_{A}& =  -\int_X \hat{A}(X)ch(S^{+}\otimes \ad P) -\frac{1}{2}(\dim \ker(L_{\alpha})-\eta_{\alpha}(0))\\
& = -\frac{1}{2}\int_X(\cL+\cE)(X)(3-8c_2(P)) -\frac{1}{2}(\dim \ker(L_{\alpha})-\eta_\alpha(0))\\
& = 8c_2(P)-\frac{3}{2}(\chi+\sigma+\eta_{\theta}(0))(X)-\frac{1}{2}h_{\alpha}+\frac{1}{2}\eta_{\alpha}(0).
\end{split}
\end{equation*}
since by the Hirzebruch signature theorem 
\begin{equation}
\int_X \dfrac{1}{3}p_1(\ad P)=\Sign(X)+\eta_{\theta}(0)
\end{equation}
and the eta invariant is 
\begin{equation}
\eta_{\alpha}(s)=\sum_{\lambda \neq 0}  \sign(\lambda)\vv \lambda\vv^{-s}
\end{equation}
where the sum is over non zero eigenvalues of $L_{\alpha}$ \cite{APS1}. The relative second Chern class is the energy on $X$:
\begin{equation}
\ell=c_2(P)=\dfrac{1}{8\pi^2}\int_X Tr(F_A^2) \equiv CS(\alpha) \mod \bZ.
\end{equation}
Putting this together we get 
\begin{equation}
8\ell-\dfrac{3}{2}(\chi+\sigma)(X)-\dfrac{1}{2}h_{\alpha}-\dfrac{3}{2}\eta_{\theta}(0)+\dfrac{1}{2}\eta_{\alpha}(0).
\end{equation}
The Atiyah-Patodi-Singer rho invariant is $\rho(\alpha)=\eta_{\alpha}(0)-3\eta_{\theta}(0)$
and this results in the following formal dimension formula
\begin{equation}
\dim \cM_{\ell}(X,\alpha)=8\ell-\dfrac{3}{2}(\chi+\sigma)(X)-\dfrac{1}{2}(h^1_{\alpha}+h^0_{\alpha})+\dfrac{1}{2}\rho(\alpha)
\end{equation}
where $h^i_{\alpha}=\dim_{\bR}H^i(\Sigma,\ad\alpha))$ for $i=0,1$. The corresponding dimension formula of a Floer-type moduli space is
\begin{equation}
\dim \cM_{\ell}(\Sigma\times \bR,\alpha,\beta)=8\ell-\dfrac{1}{2}(h_{\alpha}+h_{\beta})+\dfrac{1}{2}(\rho(\beta)-\rho(\alpha))
\end{equation}
with $h_{\alpha}=h^1_{\alpha}+h^0_{\alpha}$, and similarly for $h_{\beta}$ with $\ell \equiv CS(\beta)-CS(\alpha) \mod \bZ$.

The moduli space of interest will be those $g$-ASD connections on $X$ with finite Yang-Mills action and asymptotic to the trivial product connection on the $\End(X)$. Since the Chern-Simons invariant of the trivial connection vanishes we have moduli spaces $\cM_{\ell}(X,\theta)$ with energy $\ell \in \bZ^+$. When the intersection form is negative definite this formal dimension is given by
\begin{equation}
\dim\cM_{\ell}(X,\theta)=8\ell-3
\end{equation}  
as in the closed negative definite case. For one unit of total Yang-Mills energy this is $5$-dimensional and with our choice of a $\pi$-invariant Riemannian metric the anti-self duality equations $F_A=-\ast F_A$ are $\pi$-invariant, moreover as the action is given by orientation-preserving isometries the charge is preserved . Thus we get an induced $\pi$-action on $\cM_{1}(X,\theta)$ and this is the equivariant instanton-one moduli space we study to extract information about the original $\pi$-action on $X$. 

\section{Uhlenbeck-Taubes Compactification}

Let us recall the compactness theorem of Uhlenbeck (\cite{Lawson85}, \cite{FU91},\cite{DK90}) for the instanton moduli space over a closed, simply-connected Riemannian four manifold $(X,g)$. Intuitively, if we are given an infinite sequence of uniformly bounded, $g$-ASD connections without a convergent subsequence then there exists a gauge equivalent subsequence which has a weak limit where the limiting ASD connection has a curvature density that accumulates in integral amounts of the total energy around a finite number of points. So for a moduli space with one unit of total energy, there can be at most one point where curvature is highly concentrated, more precisely,  

\begin{thm}[Uhlenbeck,\cite{DK90}]
Let $(X,g)$ denote a smooth, simply-connected Riemannian four-manifold and let $\{A_n\}$ denote a sequence of $g$-ASD connection on a $SU(2)$-bundle with charge $c_2(P)=1$. Then there exists a subsequence also denoted by $\{A_n\}$ such that one of the following holds:
\itemize
\item For each $A_n$ there exists a gauge-equivalent connection $\widetilde{A_n}$ such that $\{\widetilde{A_n} \}$ converges in the $C^{\infty}$-topology on $X$ to a ASD connection $A \in \cM_1(X)$.
\item There exists a point $x \in X$ and trivializations $\rho_n:X-\{x\} \times SU(2) \rightarrow E\vv_{X-x}$ such that $\rho_n^{*}(A_n\vv_{X-x})$ converges in $C^{\infty}$-topology on compact subsets to the trivial product connection and the curvature densities $\vv F_{A_n} \vv^2$ converge as measures to $8\pi^2 \delta(x)$.
\end{thm}
Uhlenbeck compactness still holds in the setting of cylindrical-end four manifolds $X=X_0 \cup \Sigma \times [0,\infty) $. Suppose $\{[A_n]\} \subset \cM_1(X,\theta)$ is an infinite sequence of $g$-ASD connections on $X$ with $\sup \vv F_{A_n} \vv$ uniformly bounded. Then after passing to a subsequence again still denoted by $\{[A_n]\}$, we can still find a gauge equivalent sequence that converges on compact subsets, but since our manifold $X$ is non-compact, there is the possibility that curvature escapes down the cylindrical-end and this leads to the following weak convergence. This is the convergence without loss of energy (see \cite{MMR},6.3.3 or \cite{D02} 5.1) stated in our setting for the moduli space $\cM_{\ell}(X,\theta)$ which summarizes the compactness properties of a sequence of ASD connections. 

Let $[A_n]$ denote an infinite sequence of $g$-ASD connection in $\cM_{\ell}(X,\theta)$ with $\sup\vv F_{A_n}\vv$ uniformly bounded. Then there exists principal $SU(2)$ bundles $P(0),...,P(k)$ with $P(0) \rightarrow X$ and $P(1),...,P(k) \rightarrow \Sigma \times \bR$ together with $L^2$-finite $g$-ASD connections $A(i)$ on $P(i)$. 

There exits a sequence of positive real numbers $\{T_n(i)\}$ for $i=1...k$. with the following properties
\begin{itemize}
\item $T_n(1) < T_n(2) < ... < T_n(k)$ for all $n$,
\item $\lim_{n\rightarrow \infty} T_n(i) = \infty$ for $i=1,...,k$,
\item $\lim_{n\rightarrow \infty} (T_n(i)-T_n(i-1)) =\infty$ for $i=2,...,k$. 
\end{itemize}
\noindent
such that the following holds
\begin{enumerate}

\item $[A_n]$ converges in the $C^{\infty}$-topology on compact subsets of $X$ to $[A(0)]$ and the curvature densities $\vv \Tr F_{A_n}^2 \vv$ converge in the sense of measures on compact subsets to $\vv \Tr F_{A(0)}^2 \vv$.

\item the translations $[c_{T_n(i)}^{\ast} A_n\vv_{\Sigma \times \bR}]$ converge in the $C^{\infty}$-topology on compact subsets of $\Sigma \times \bR$ to $[A(i)]$ and the curvature densities $\vv \Tr F_{c_{T_n(i)}^{\ast} A_n}^2 \vv$ converge in the sense of measures on compact subsets to $\vv \Tr F_{A(i)}^2 \vv$ for $i:1,...,k$. Here $c_T:\Sigma \times \bR \rightarrow \Sigma \times \bR$ is the translation by $T$ in the ``time" factor.

\item (no loss of energy) $\ell=\sum_{i=0}^{k} \ell_i$ where $\ell_i=c_2(A(i))$ the second relative Chern number, with $\ell_0 \equiv CS(A(0)) \mod \bZ$ and $\ell_i \equiv CS(A(i+1))-CS(A(i))\mod \bZ$.

\item The limiting flat connections have compatible boundary values  $\partial_{\infty}([A(0)])=\partial^{-}_{\infty}([A(1)])$ and $\partial^{-}_{\infty}([A(i)])=\partial^{+}_{\infty}([A(i+1)])$ and $[A_k]=\theta$.

\end{enumerate}

Weak limits are defined as a tuple of gauge equivalence class of $L^2$-finite ASD connections $[A]:=([A_0],[A_1],\cdots,[\theta])$ where  $[A_0]\in \cM_{\ell_0}(X,\alpha_0)$ and $[A_i]\in \cM_{\ell_{i}}(\Sigma \times \bR,\alpha_{i-1},\alpha_i)$, $\alpha_i$ are flat connections on $\Sigma$ and have compatible boundary values $\partial_{\infty}(A_i)=\partial_{\infty}(A_{i+1})$. So the ``ends" of the moduli space $\cM_{1}(X,\theta)$ are parametrized by products of the form 
\begin{equation}
\cM_{\ell_0}(X,\alpha_0) \times \cM_{\ell_1}(\Sigma \times \bR, \alpha_0,\alpha_1) \times ... \times \cM_{\ell_k}(\Sigma \times \bR, \alpha_{k-1}, \theta).
\end{equation}
with $\sum_i \ell_i =1$.

The question of which points in $X$ parametrize the centers of highly concentrated curvature of ASD connections is provided by the Taubes map which constructs nearly anti-self dual connections on $X$ by grafting the standard instanton on the four sphere $S^4$ to the trivial product connection on $X$ using a cut-off function. This process disturbs anti-self duality but the self-dual part acquires a small uniform bound on the curvature and this allows for a small perturbation so that the resulting connection is anti-self dual. We outline the construction below, see \cite{T87} for details and also \cite{BKS90} for the equivariant case. 

Choose an oriented orthonormal frame at a point $x\in X$ to identify a neighborhood of $x$ with a small ball of radius $\lambda_0$. Since $X$ has bounded geometry we can make this choice of $\lambda_0$ independent of the point by letting it be less than the injectivity radius of $X$. Define a degree one map $f: X \rightarrow S^4$ which maps the point $x$ to the north pole and collapses everything outside the ball to the south-pole. Take the standard instanton $I$ on a principal $SU(2)$ bundle $Q$ with $c_2(Q)=1$ with concentration $1/\lambda$ at the north-pole for $\lambda \in (0,\lambda_0)$. Pulling back $I$ and extending it using a cut-off function to be trivially flat outside of the ball neighborhood gives a connection $A(x,\lambda)$ which satisfies the following bounds:
\begin{equation}
\Big( \int_X e^{\tau \delta} \vv F_A \vv^p \Big) ^{1/p} \leq (\text{const}) e^{\tau \delta/p}\lambda^{4/p-2}
\end{equation}
and the self dual part
\begin{equation}
\Big(\int_X  e^{\tau \delta}\vv F_A^{+} \vv^p \Big)^{1/p} \leq (\text{const}) e^{\tau \delta/p}\lambda^{2/p}
\end{equation}
a different choice of orthonormal frame gives a gauge-equivalent connection. The Taubes map is given by $T: X \times (0,\lambda_0) \rightarrow \cB(P)$ sending $(x,\lambda)$ to $[A(x,\lambda)]$. Consider a perturbation term $A(x,\lambda)+a(x,\lambda)$ with the condition that $F_{A+a}^{+}=0$, this leads to the following equation
\begin{equation}
0=F_A^{+}+d_A^{+}a+(a\wedge a)^+
\end{equation}
and since the operator $d_A^+$ is not elliptic, the condition $a=d_A^{*}u$ for some $u \in \Omega_X^{+}(\ad P)$ is imposed. Now the equation to be solved is:
\begin{equation}
d_A^{+}d_A^{*}u=-F_A^{+}-(d_A^{*}u \wedge d_A^{*}u)^+
\end{equation}
for a smooth solution $u$ which then gives an ASD connection $\widetilde{A}(x,\lambda)=A(x,\lambda)+a(x,\lambda)$ where $a=d_A^*u$. Note that the linearized equation $d_A^+d^{\ast}_A u =-F_A^+$ has a solution if and only if the first eigenvalue of $d^+_Ad^{\ast}_A$ is positive (\cite{Lawson85},\cite{FU91}). It is at this stage that the intersection form of $X$ being negative definite plays a role. The original equation is solved by an iterative scheme (see \cite{Lawson85}, \cite{taubes82}). Define a map $\widetilde{T}: X \times (0, \lambda_0) \rightarrow \cM_1(X,\theta)$ sending $[A(x,\lambda)]$ to $[\widetilde{A}(x,\lambda)]$. 
\begin{prop}[\cite{T87}]
There exists $\delta_1 > 0$ such that for any $\delta \in (0,\delta_1)$, the moduli space $\cM_1(X,\theta)$ is non-empty. There is an open set $K \subset \cM_1(X,\theta)$ such that for small $\lambda_0$, $K$ is diffeomorphic to $X \times (0,\lambda_0)$ and isotopic in $\cB$ to the image of the Taubes map $\widetilde{T}$.
\end{prop}
Here $\delta$ denotes the $\delta$-decay used in the Sobolev completion. Since $\pi$ acts by orientation preserving isometries and the metric is $\pi$-invariant we have by construction
\begin{equation}
t \cdot \widetilde{T}(x,\lambda)=\widetilde{T}(t \cdot x,\lambda).
\end{equation}
This gives an equivariant collar in the moduli space and a partial compactification 
\begin{equation}
\overline{\cM}_1(X,\theta)=\cM_1(X,\theta) \cup X \times (0,\lambda_0)
\end{equation}
consisting of ASD connections with highly concentrated curvature. In particular, the equivariant moduli space $\cM_1(X,\theta)$ is non-empty when $X^\pi$ is non-empty. For connections $[A] \in X \times (0,\lambda_0)$ Taubes also gives that $H_A^2=0$ \cite[Theorem 3.38, pg. 81]{Lawson85} so that a neighborhood of the collar is smooth $5$-manifold and these connections are irreducible. The fixed set $X^{\pi}$ give rise a family of ASD connections which correspond to an equivariant lifts of the $\pi$-action on $X$ to a $\widetilde{\pi}=\bZ/2p$-action on the principal $SU(2)$-bundle. 

\begin{figure}[ht]
\centering
\includegraphics[scale=.45]{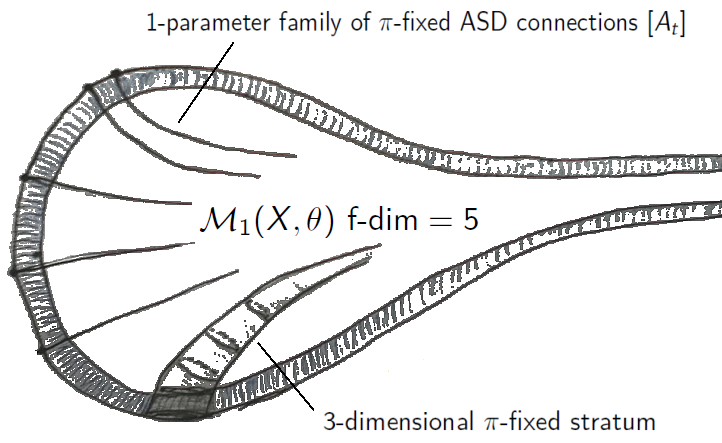}            
\caption[The equivariant instanton-one moduli space on cylindrical-end four manifolds.]{The instanton one moduli space $\cM_1(X,\theta)$ of ASD connections asymptotic to the trivial flat connection and one unit of total Yang-Mills energy has a collar end as in the closed definite case. The fixed set propagates a family of fixed ASD connections into the moduli space, each corresponding to equivariant lifts of the $\pi$-action to the principal $SU(2)$-bundle which leave a one-parameter(respectively three-parameter) family of connections $\pi$-invariant. The equivariant lifts can be determined by pulling-back an equivariant bundle on the four sphere at the ideal boundary.}
\end{figure} 

We give an outline below, for more details see \cite{BM93} or \cite{HL92}. Let $\Diff(X)$ denote the group of diffeomorphisms of $X$ then there is an exact sequence:
\begin{equation}
1 \rightarrow \cG(P) \rightarrow \Aut(P) \rightarrow \Diff(X)
\end{equation}
where $\Aut(P)$ are bundle automorphisms of $P$ which do not necessarily cover the identity. A lift of the $\pi$-action to $P$ is a homomorphism $\pi \rightarrow \Aut(P)$ which projects back to $\pi$ under the above map to $\Diff(X)$. Let $\cG(\pi) \subset \Aut(P)$ denote the bundle automorphisms which cover the $\pi$-action on $X$, then we also have an exact sequence:
\begin{equation}
1 \rightarrow \cG(P) \rightarrow \cG(\pi) \rightarrow \pi \rightarrow 1
\end{equation}
The natural action of $\cG(\pi)$ on $\cA(P)$ is well-defined up to gauge transformations so we get an induced action of $\cG(\pi)/\cG(P)=\pi$ on $\cB(P)$. Let $[A]$ denote a $\pi$-fixed connection in $\cB(P)$ then the following sequence is exact:
\begin{equation}
1 \rightarrow \cG_A \rightarrow \cG_A(\pi) \rightarrow \pi \rightarrow 1
\end{equation}
where $\cG_A$ denotes the stabilizer of $A$ in the gauge group and $\cG_A(\pi)$ denotes the stabilizer in $\Aut(P)$. There then exists a lift of the action on $X$ to the principal bundle leaving the connection $A$ invariant if and only if the above sequence splits. For irreducible $SO(3)$ connections the existence and uniqueness of a lift follows since $\cG_{A}$ is trivial, but for irreducible $\pi$-fixed connections $[A]$ we have $\cG_A=\{\pm 1\}$ and so we have $\cG_A(\pi)$ is either $\bZ/2p$ or $\bZ/p \times \bZ/2$. So for $p$ is odd we can consider the double cover $\widetilde{\pi}=\bZ/2p$ acting on the bundle which covers the $\pi$-action on $X$ and leaves the connection $A$ $\pi$-invariant. Choosing a different representative in $[A]$ gives an equivalent lift and if $I$ is a set parametrizing equivalence classes of lifts then we get disjoint union of the fixed set 
\begin{equation}
\Fix(\cB^{*},\pi)=\bigsqcup_{i \in I} \cA_i^{*}/\cG_i=\bigsqcup_{i \in I}\cB^{*}_i
\end{equation} 
where $\cA_i$ are $i$-invariant connections and $\cG_i$ are $i$-invariant gauge transformations. The fixed sets may intersect at a reducible where $\cG_A \neq 1$ there may be more than one lift of the action leaving $A$ invariant. In our application there are no reducible ASD connections \cite[p. 403]{T87}.
\begin{thm}[\cite{F89},\cite{BM93}]
The image of $\cB_i$ in $\cB$ is closed, $\cB^{*}_i$ is a closed smooth submanifold of $\cB^{*}$.
\end{thm}
\noindent
We will study the $\pi$-equivariant compactification of the fixed-set $\cM_1(X,\theta)^{\pi}$ which originates in the Taubes collar to obtain information about the original $\pi$-action on $X$.

\section{Equivariant General Position}

In the non-equivariant setting, the argument in Freed and Uhlenbeck can be adapted to show that for a Baire set of metrics $g$ which restrict to a product metric on the $\End(X)$ the moduli space $\cM_1(X,\theta)$ is a smooth $5$-dimensional manifold. In the equivariant setting we have a theorem of Cho \cite{cho91} on the existence of a Baire set of $\pi$-invariant metrics on $X$ such that all the components of the fixed set $\cM_1(X,\theta)^{\pi}$ are either empty or smooth manifolds. This $\pi$-invariant version of Freed and Uhlenbeck is also valid on cylindrical-end $4$-manifolds (it was used in \cite{BKS90}). Even though $\cM_1(X,\theta)^{\pi}$ is smooth if non-empty, it may not have smooth $\pi$-invariant neighborhoods and in general the surrounding moduli space can be highly singular. The next possible approach would be to perturb the anti-self duality equations at the chart level by passing to the Kuranashi model
\begin{equation}
\phi: H_A^1 \rightarrow H^2_A 
\end{equation}
as in Donaldson \cite{D83}. In the equivariant case $H_A^1$ and $H_A^2$ are finite dimensional real $\pi$-representation spaces and the obstruction to the existence of an equivariant perturbation is 
\begin{equation}
[H_A^1] - [H^2_A] \in R^+(\pi) 
\end{equation}
being an actual real representation. Hambleton and Lee in \cite{HL92} applied the theory of equivariant general position of Bierstone \cite{Bie77} to equivariant moduli spaces. Bierstone's approach is to first consider defining the problem of equivariant transversality at the origin of $W$ with respect to $0 \in V$ of a smooth, $f:V \rightarrow W$ two $G$-vector spaces $V$ and $W$. The space $\cC^{\infty}_{G}(V,W)$ of smooth $G$-equivariant maps is a module over the ring $\cC^{\infty}_{G}(V)$ of smooth $G$-invariant functions on $V$. Then there exists a finite set of polynomial generators $F_1,\cdots F_k$ of $\cC^{\infty}_{G}(V,W)$ so that
\begin{equation}
f(x)=\sum^{k}_{i=1} h_i(x)F_i(x)
\end{equation}
where $h_i(x) \in \cC^{\infty}_{G}(V)$. We can write this slightly differently as 
\begin{equation}
f(x)=U\circ \text{graph}(h(x))
\end{equation}
where $U: V \times \bR^{k} \rightarrow \bR$ defined by $U(x,h)=\sum_{i=1}^{k}h_iF_i(x)$  and $\text{graph}(h):V\rightarrow V \times \bR^k$ defined by $\text{graph}(h(x))=(x,h_1(x),\cdots, h_k(x))$. Then $f^{-1}(0)=U^{-1}(0) \cap \text{graph}(h)$.
The affine algebraic variety $U^{-1}(0)$ has a natural $\pi$-equivariant minimum Whitney stratification. Let $f:V \rightarrow W$ be a smooth map between $G$-vector spaces, then $f$ is in \textbf{$G$-equivariant general position} with respect to $0 \in W$ at $0 \in V$ if the graph of $h$ is stratum-wise transverse to the affine algebraic variety $U^{-1}(0)$ at $0 \in V$. This notion is well-defined in the sense that it does not depend on the choice of generators $F_i$ and $h_i$.  The subspace of smooth equivariant maps in general position with respect to $0$ is open and dense in $\cC^{\infty}_{G}(V,W)$ with respect to the $C^{\infty}$-topology. This result is extended to the infinite dimensional setting by Hambleton and Lee \cite{HL92}.
%
In our case we adapt Floer's method \cite[2c]{Floer88} to make equivariant perturbations of the ASD equations over a cylindrical-end four manifold $X$ to obtain moduli spaces in Bierstone general position (see \cite{HL92} for the case of a closed 4-manifold $X$ and chart-by-chart perturbations).

For these we use Wilson loop perturbations in free $\pi$-orbits of embedded circles in $X$. The non-equivariant case is described in \cite[p. 400-401]{Don87}. We will review the construction of these equivariant perturbations below with notation from \cite[p. 129-130]{Nikolai}. Let $\gamma: S^1 \times D^3 \rightarrow X$ be an embedded loop in $X$ which is slightly thickened. Given a connection $A$ and $x\in \gamma(S^2 \times D^3)$ let $\Hol_A(\gamma,x)$ denote the holonomy around the loop $\gamma_x$ parallel to $\gamma$. If we denote by $\Pi: SU(2) \rightarrow \underline{\textit{su}}(2)$ the map $u \mapsto u-\dfrac{1}{2}tr(u)\Id$ then $\Pi \Hol_A(\gamma)$ defines a section on $\ad(P)$ over $\gamma(S^1\times D^3)$. If $\omega$ is $2$-form compactly supported in $\gamma(S^1\times D^3)$ then 
\begin{equation}
\eta(\omega,\gamma,A)=\omega \otimes \Pi \Hol_A(\gamma) \in \Omega^{2}(X,\ad P)
\end{equation}
Now given a finite set of embedded loops $\gamma_i$ and $2$-forms $\omega_i$ for $i=1\ldots m$ then define a linear combination
\begin{equation}
\sigma(A)=\sum_i^m \varepsilon_i \eta(\omega_i,\gamma_i,A) 
\end{equation}
and consider the perturbed ASD equations
\begin{equation}
F_A=-\ast F_A + \sigma_{+}(A).  
\end{equation}
where $\sigma_{+}(A)$ is the orthogonal projection onto $\Omega_+^2(X,\ad P)$. In the equivariant setting we use $\pi$-orbits of $m$ freely embedded loops $\gamma_i$ and consider 
\begin{equation}
\sigma(A)=\sum_{i=1}^{m}\sum_{s \in \pi} \varepsilon_i \eta((s^{-1})^{\ast}\omega_i,s(\gamma_i),A)
\end{equation} 
and define 
\begin{equation}
\widehat{\sigma}(A)=\sum_{t \in \pi} (t^{-1})^{\ast}\sigma(t^{\ast}A).
\end{equation}
The perturbed section $F_A^+ + \widehat{\sigma}_+(A)$ is now $\cG(\pi)$-equivariant and so the perturbed moduli space inherit a $\pi$-action as before.

Since Bierstone general position is an open-dense condition, a generic equivariant perturbation of the ASD equations give the moduli spaces the structure of a Whitney stratified space (see \cite{Bie77} or \cite{HL92} for details). 

\section{Proof of Theorem A}

We begin with a lemma to determine the equivariant bundle structure using the Taubes map in terms of the weights of the isotropy representation over the fixed-points. These weights $\pm \lambda$ are given by a representation $\widetilde{\pi} \rightarrow SU(2)$ over a fixed point:
\begin{equation}
\widetilde{t} \mapsto \begin{pmatrix} \widetilde{t}^{\lambda} &  \\  & \widetilde{t}^{-\lambda} \end{pmatrix}
\end{equation}
for a lift of the $\pi$ action on $X$ to a $\widetilde{\pi}=\bZ/2p$ action on the principal $SU(2)$ bundle, where $\widetilde{t}$ is the generator.

\begin{lem}
If a fixed point has rotation numbers $\bC^2(a,b)$ then the the equivariant lift it generates in the moduli space has isotropy representations over the fiber of this point given by $\bZ/2p$-weights $\pm (b-a)$ and over the other fixed points $\pm (a+b)$. Similarly for the fixed $2$-spheres.
\end{lem}
\begin{proof}
Choose a $\pi$-invariant disk around the fixed point and use the map $f:X \rightarrow S^4$ in the Taubes construction. We can then pull-back a equivariant $\widetilde{\pi}$-bundle $Q$ over the four sphere with $c_2(Q)=1$ \cite{FL86}, together with its one-parameter family of $\pi$-invariant instantons. 
\end{proof}
Suppose there are at least three fixed points of the $\pi$-action $p_i$, say with rotation numbers $(a_1,b_1), (a_2,b_2), (a_3,b_3)$. Each of these fixed points lies at the Taubes collar of the moduli space and is part of a $\pi$-fixed arc $\gamma_i$. We would like to show that none of these arcs can connect with each other in the irreducible component of the moduli space.
\begin{lem}[\cite{HL95}]
If the $\pi$-action has at least $3$ fixed points then the $\gamma_i$ represent distinct equivariant bundle structures and are therefore disjoint in $\cM^{\ast}_1(X,\theta)$.
\end{lem}
\begin{proof}
Suppose $\gamma$ connects $p_1$ and $p_2$, the normal bundle information is propagated along this oriented arc and gives a canceling pair so that $(a_2,b_2)=(a_1,-b_1)$. We will use the presence of the third distinct fixed point $p_3$ to get a contradiction. Because the point $p_1$ is fixed, there is an $\pi$-invariant ball $B(p_1)$ with a linear action so this allows us to construct an equivariant degree one map $f_1:X \mapsto S^4$, now we can pull-back the equivariant bundle structure $Q \mapsto S^4$ via $f_1$ and get an equivariant bundle $(X,f_1^{\ast}Q)$. Similarly, we can do this with a map $f_2$ about the point $p_2$, this gives an equivariant bundle structure $(X,f_2^{\ast}Q^{\prime})$. Since these bundle structures are equivalent, the isotropy at $p_3$ has to agree and a comparison shows that either $2a\equiv 0 \pmod p$ or $2b\equiv 0 \pmod p$, in either case we get a contradiction. 
\end{proof}
\noindent
The following lemma deals with the case of a fixed $2$-sphere and its $3$-dimensional stratum.
\begin{lem}
If a fixed $2$-sphere in $X_0$ represents a non-trivial homology class with non-zero self-intersection, then the $3$-dimensional $\pi$-fixed stratum generated at the Taubes boundary cannot bound off in $\cM_1^{\ast}(X,\theta)$.
\end{lem}
\begin{proof}
A fixed $2$-sphere $S$ in $X$ represents a non-trivial homology class $[S]\in H_2(X,\bZ)$. Let $F$ denote the $3$-dimensional stratum in the moduli space which arises from the Taubes boundary and suppose $\partial F=S$. Let $[c]=\mu([S]) \in H^{2}(\cB^{*},\bZ)$ and $i:X \rightarrow \cM_1(X,\theta)$ denote the inclusion map, then $i^{*}\mu([S])$ is the Poincar\'e dual $PD([S]) \in H^{2}(X,\bZ)$ (\cite{DK90} 5.3)and so $\langle i^{*}c,\partial[F]\rangle$ evaluates non-trivially. On the other hand, we have $\langle i^{*}c,\partial[F]\rangle=\langle i^{*}\delta(c),[F]\rangle=0$ giving a contradiction. Thus $F$ cannot bound off in the moduli space.
\end{proof}
The previous lemmas indicate that the fixed set must have an end that is not part of the Taubes collar and according to the Uhlenbeck compactness result applicable here, must lead to energy splitting down the cylindrical end. We first rule out the case of a trivial splitting:
\begin{lem}
\label{lemma: lemmaA}
If the $\pi$-action has at least $3$ fixed points, then one-dimensional fixed set generated by the fixed points in the Taubes boundary $X\times (0,\lambda_0)$ cannot split energy in the equivariant compactification of $\cM_1(X,\theta)$ by $\cM_0(X,\theta) \times \cM_1(\theta,\theta)$.
\end{lem}
\begin{proof}
The idea is that $\cM_0(X,\theta)$ has zero energy, so it leaves behind a flat equivariant bundle which identifies the isotropy over the fibers of each fixed point. Suppose $\gamma$ is a one-parameter family of $\pi$-fixed ASD connections generated at the Taubes boundary from the fixed point with rotation numbers $(a,b)$. Then the corresponding equivariant lift has isotropy over the fiber of this fixed point with weight $\pm(b-a)$ and $\pm(a+b)$ over the other fixed points. In such a energy splitting a flat equivariant bundle identifies the isotropy over all the points, so $a+b=\pm (b-a)$ and this forces either $2a\equiv 0 \pmod p$ or $2b\equiv 0 \pmod p$. Since $p$ is odd and $(a,b)$ are rotation numbers for a fixed-point we get a contradiction.
\end{proof}
According to the above lemmas, the equivariant compactification of the fixed-set $\cM_1(X,\theta)^{\pi}$ that arises from the Taubes boundary $X^{\pi}\times (0,\lambda_0)$ must limit to a connection that takes energy down the cylindrical-end $\Sigma(2,3,5) \times [0,\infty)$ and splits non-trivial energy. This will involve the flat connections of $\Sigma(2,3,5)$, 
which as representations $\alpha$ of the fundamental group into $SU(2)$ are determined by rotation numbers $(\ell_1,\ell_2,\ell_3)$ and for $\Sigma(2,3,5)$ there are only two irreducible representations \cite{Nikolai}. We record here a table that gives the necessary values for index calculations. 
\begin{table}[ht]
\begin{center}
\small
\begin{tabular}{|l|l|l|l|l|}\hline
$\alpha$ & $(\ell_1,\ell_2,\ell_3)$ & $\mu(\alpha)$ & $\rho(\alpha)/2$ & $-CS(\alpha) \in (0,1]$\\ \hline
1 & (1,2,2) & 5 & -97/30 & 49/120\\ \hline
2 & (1,2,4) & 1 & -73/30 & 1/120\\  \hline
\end{tabular}
\end{center}
\caption{For each flat connection $\alpha$ of $\Sigma(2,3,5)$ are listed values for the Floer $\mu$-index modulo $8$, one-half the Atiyah-Patodi-Singer $\rho$-invariant and minus the Chern-Simons invariant of the given flat connection (\cite{FS90}). The values for the $\rho$-invariant can be computed using a flat $SO(3)$-cobordism to a disjoint union of lens spaces (see \cite[p. 144]{Nikolai}}.
\label{table: TableA1}
\end{table}

The energy in $\cM(\alpha_i,\theta)$ is given by $-CS(\Sigma(2,3,5),\alpha_i) \mod \bZ \in (0,1]$ (\cite{FS90}, \cite[p. 101]{Nikolai}). In an energy splitting, the moduli space has an end given by a local diffeomorphism
\begin{equation*}
\cM_{\ell_0}(X,\alpha_0) \times_{\alpha_0} \cM_{\ell_1}(\alpha_0,\alpha_1) \times_{\alpha_1} \cdots \times_{\alpha_{k-1}} \cM_{\ell_k}(\alpha_{k-1},\theta) \rightarrow \cM_1(X,\theta)
\end{equation*}
where $\{\alpha_i\}_{i=1}^{k-1}$ are irreducible flat connections on $\Sigma(2,3,5)$ -- this then leads to a dimension count
\begin{equation*}
5=\dim \cM_{\ell_0}(X,\alpha_0)+\sum_{i=1}^{k} \dim \cM_{\ell_{i}}(\alpha_{i-1},\alpha_{i})
\end{equation*}
with $\alpha_{k}=\theta$ and as the convergence is without loss of energy we get the condition $\sum_{i=0}^{k} \ell_{i}=1$. The dimensions can be determined modulo $8$ by the formulas
\begin{equation}
\dim \cM(\alpha,\beta) \equiv \mu(\alpha)-\mu(\beta)- \dim \text{Stab}(\beta) \pmod 8
\end{equation}
and $\dim \cM(X,\alpha)\equiv -\mu(\alpha)-3 \pmod 8$ \cite{Floer88} where $\mu$ is the Floer index and $\mu(\theta)=-3$. Imposing the energy condition allows one to determine the exact geometric dimensions and since there are only $2$ irreducible flat connections on $\Sigma(2,3,5)$ denoted by $\alpha_1=(1,2,2)$ and $\alpha_2=(1,2,4)$ we have only the possibilities in Table \ref{TableA2}.

\begin{table}[ht]
\begin{center}
\small
\begin{tabular}{|l|l|l|l|}\hline
 & \small{Charge-Splitting} & \small{Dimension} & \small{Energy}\\\hline
A & $\cM(X,\alpha_1)\times \cM(\alpha_1,\theta)$ & \small{0+5} & \small{71/120+49/120=1}\\\hline
B & $\cM(X,\alpha_1)\times \cM(\alpha_1,\alpha_2)  \times \cM(\alpha_2,\theta)$ & \small{0+4+1} & \small{71/120+2/5+1/120=1}\\\hline
C & $\cM(X,\alpha_2)\times \cM(\alpha_2,\theta)$ & \small{4+1} & \small{119/120+1/120=1}\\\hline
D & $\cM(X,\theta)\times \cM(\theta,\theta)$ & \small{0+5} & \small{0+1}\\\hline
 \end{tabular}
 \end{center}
\caption{All possible energy splitting in the compactification of $\cM_1(X,\theta)$.}
\label{TableA2}
\end{table}

\begin{figure}[ht]   
\centering
\includegraphics[scale=.45]{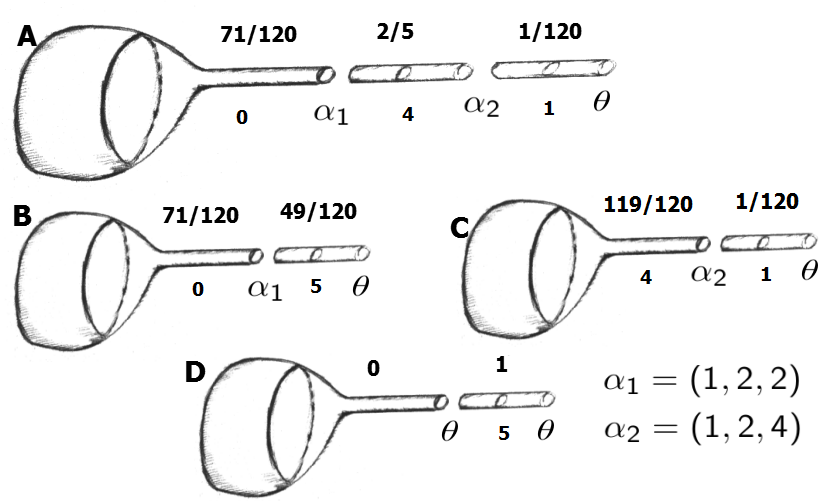}            
\caption[Charge splitting for the Poinar\'e homology $3$-sphere.]{The figure shows the charge splitting that can occur for $\Sigma(2,3,5)$. The number above is the Yang-Mills energy and is given by $CS(\beta)-CS(\alpha) \pmod \bZ$ put in the range $(0,1]$. The value below is the formal dimension of ASD connections with the corresponding fixed energy. As energy is conserved, the total sum of the energy values is one and the sum of the formal dimensions is 5.}
\end{figure} 

Let $\Sigma(b=0,(a_1,b_1),(a_2,b_2),(a_3,b_3))$ be the Seifert invariants $\Sigma(a_1,a_2,a_3)$ and $\pi=\bZ/p$ act as the standard action on $\Sigma(a_1,a_2,a_3)$. Then the quotient $Q=\Sigma/\pi$ is a rational homology sphere(or a $\bZ$-homology lens space) with Seifert invariants $Q(b=0;(a_i,pb_i))$; we will need a formula for the Chern-Simons invariant of reducible flat connections on $Q$. Note that if we take the $p$-fold cover we get the trivial product connection on $\Sigma(a_1,a_2,a_3)$ and as Chern-Simons is multiplicative under finite covers we expect an expression of the form $$CS(Q,\alpha(k))\equiv \frac{n}{p} \mod \bZ$$ for some integer $n$. That this is indeed the case can be verified using Auckly's formula in \cite{Auckly94b}:
\begin{thm}
$CS(Q,\alpha(k))\equiv \dfrac{n_0k}{p}\pmod \bZ$ where $n_0$ satisfies $n_0a_1a_2a_3 \equiv k \pmod p$.
\end{thm}
\begin{proof}
This is obtained by using a representation $\rho(n_0,n_1,n_2,n_3): \pi_1(Q) \rightarrow U(1)$ where $n_0$ satisfies $a_1a_2a_3\cdot n_0 \equiv k \pmod p$ and $n_i=b_i$ for $i\neq 0$. This representation sends $h \rightarrow e^{2\pi i k/p}$ and $x_i \mapsto 1$. The Seifert invariants satisfy 
\begin{equation}
\sum_i \dfrac{b_i}{a_i}=\dfrac{p}{a_1a_2a_3}
\end{equation}
and the formula for the Chern-Simons invariant of the corresponding flat connection is given in \cite[p.234]{Auckly94b} as 
\begin{align*}
CS(Q,\rho)\equiv- & \sum_{j=1}^{3} \dfrac{\rho_jn_j^2+n_j(n_0+c/2+\sum_{i=1}^{3}n_i/a_i)/(b+\sum_i b_i/a_i)}{a_j}\\
          &-\dfrac{(n_0+c/2)(n_0+c/2+\sum_i n_i/a_i)}{b+\sum_i b_i/a_i} \pmod \bZ
\end{align*}
with $c=0$ and $\rho_j$ satisfies $a_j\sigma_j-b_j\rho_j=1$ for some integers $\sigma_j$. This then simplifies to $\dfrac{n_0k}{p} \mod \bZ$.
\end{proof}
The Chern-Simons invariants for the irreducible flat connections on $Q$ can be computed by using an $SO(3)$ flat cobordism to a disjoint union of lens spaces as with $\Sigma(a_1,a_2,a_3)$ but also again from \cite[p.232]{Auckly94b}. 
\begin{thm}[\cite{Auckly94b}]
For an irreducible flat $SU(2)$ connection $\alpha$ on $Q$ 
\begin{equation} 
  CS(Q,\alpha)\equiv - \sum_{i=1}^{3}(\frac{\rho_i {\ell_i}^2}{a_i} +\frac{\ell_i}{a_i}) + \frac{1}{4}(b + \sum_{i=1}^3 \frac{pb_i}{a_i}) \pmod \bZ
\end{equation}
where $\rho_i$ satisfies $a_i \sigma_i - (pb_i)\rho_i=1$ for some integers $\sigma_i$, $i=1..3$.
\end{thm}
We now investigate whether any of the charge splittings given in Table \ref{TableA2} contain $\pi$-fixed ASD connections:
\begin{equation}
\cM_{\ell_0}^{\pi}(X,\alpha_0) \times_{\alpha_0} \cM_{\ell_1}^{\pi}(\alpha_0,\alpha_1) \times_{\alpha_1} \cdots \times_{\alpha_{k-1}} \cM_{\ell_k}^{\pi}(\alpha_{k-1},\theta) \rightarrow \cM_1^{\pi}(X,\theta)
\end{equation}
It follows from immediately from Lemma \ref{lemma: lemmaA} that case D is ruled out. We now rule out the possibility of a $1$-dimensional fixed set in the equivariant moduli space $(\cM_1(X,\theta),\pi)$ splitting in case B of Table \ref{TableA2}.
\begin{lem}
The moduli space $\cM_{\ell}(\alpha_1,\alpha_2)$ does not support $\pi$-fixed ASD connections with energy $\ell=2/5$ for any odd prime $p \geq 7$.
\end{lem}
\begin{proof}
If there exists a $\pi$-fixed ASD connection with energy $\ell=2/5$ in $\cM_{\ell}^{\pi}(\alpha_1,\alpha_2)$ then it corresponds to an equivariant lift of the $\pi$-action to the principal bundle which leaves that connection invariant. Since a $\pi$-invariant connection descends to an $SO(3)$ connection on the cylinder $Q \times \bR$ where $Q=\Sigma(2,3,5)/\pi$ is a rational homology $3$-sphere, the moduli space in the quotient must be non-empty. Let $\alpha_1^{\prime}$ and $\alpha_2^{\prime}$ denote the irreducible limiting flat connections on $Q \times \bR$. The connection in the quotient has energy or Pontryagin charge $4\ell/p=8/5p$, however, a non-empty moduli space must have energy that is congruent modulo $4\bZ$ (cf. \cite[Remark 5.6, p. 102]{Nikolai}) to the difference of the $SO(3)$ Chern-Simons invariants $CS(Q,\alpha_2^{\prime})-CS(Q,\alpha_1^{\prime})$. It follows from Auckly's formula (Theorem 2.4.7) that this difference has the form $n/30$ for some integer $n$. But now $\dfrac{n}{30}$ is not congruent to $\dfrac{8}{5p} \mod 4\bZ$ since if two rational numbers are congruent modulo integers then they must have the same denominator, but the former has denominator at most $30$ and for the latter $p\geq7$.  It must be then that $\cM_{\ell}^{\pi}(\alpha_1,\alpha_2)$ is empty.
\end{proof}
It remains then to investigate the remaining cases $\cM_{\ell}(\alpha_i,\theta)$; a similar argument as above provides the following more general proposition which will give us a necessary condition for the existence of $\pi$-invariant ASD connections:
\begin{prop}
Suppose a principal $SU(2)$ bundle over $\Sigma(a_1,a_2,a_3) \times \bR$ admits $\pi$-invariant ASD connections with energy $\ell \equiv \dfrac{e^2}{4a_1a_2a_3} \in (0,1]$ asymptotic to an irreducible flat connection $\alpha$ at $-\infty$ and the trivial product at $+\infty$. Then this connection descends to an $SO(3)$ ASD connection on the quotient $Q \times \bR$ with energy $4\ell/p$ which limits to an irreducible connection still denoted by $\alpha$ at $-\infty$ and a flat $U(1)$-reducible connection $\beta$ at $+\infty$ which has $SO(3)$-holonomy number $\pm e \pmod p$. 
\end{prop}
\begin{proof}
Since an invariant connection descends to a $SO(3)$ ASD connection, the moduli space in the quotient is non-empty, this again gives the relation between the $SO(3)$ Chern-Simons invariants
\begin{equation}
CS(Q,\beta)-CS(Q,\alpha) \equiv \dfrac{4\ell}{p}\equiv \dfrac{e^2}{pa_1a_2a_3} \mod 4\bZ
\end{equation}
But the Chern-Simons invariant of the reducible connection is given by $CS(Q,\beta(k))\equiv \dfrac{n_0k}{p}$ for some integer $n_0$ such that $n_0(a_1a_2a_3) \equiv k \pmod p$ and where $k$ is the $SO(3)$ holonomy number of the representation $\beta(k)$. On the other hand we have $CS(Q,\alpha) \equiv \dfrac{m}{a_1a_2a_3}$ for some integer $m$. Taking the difference gives 
\begin{align}
&\dfrac{n_0k(a_1a_2a_3)-mp}{p(a_1a_2a_3)} \equiv \dfrac{e^2}{p(a_1a_2a_3)} \mod 4\bZ
\end{align}
This implies that the numerators are congruent modulo $4p(a_1a_2a_3)\bZ$ and so gives 
\begin{equation}
k^2 \equiv e^2 \pmod p
\end{equation}
since $\bZ/p$ has no zero divisors completes the proof.
\end{proof}
\noindent
This proposition gives a necessary condition for $\Sigma(a_1,a_2,a_3) \times \bR$ with flat limits an irreducible at $-\infty$ and $\theta$ at $+\infty$ to admit invariant ASD connections: the numerator in the energy must be a square integer.

The irreducible flat connections $\alpha_1$ and $\alpha_2$ on $\Sigma(2,3,5)$ descend to irreducible flat connections on the quotient $\Sigma(2,3,5)/\pi$ which we still denote by $\alpha_i$.
\begin{thm}
Let $Q$ denote the rational homology sphere quotient $\Sigma(2,3,5)/\pi$ and $\ell=49/120$. Then the formal dimension of the moduli space $\cM_{4\ell/p}(Q\times \bR,\alpha_1,\beta)$ of $SO(3)$-ASD connections on the cylinder $Q\times \bR$ with energy $4\ell/p$ that limit to $\alpha_1$ at $-\infty$ and to a reducible connection $\beta$ at $+\infty$ is 1 when the holonomy representation of the flat connection $\beta$ is $\pm 7 \pmod p$. Similarly, when $\ell=1/120$ the formal dimension of $\cM_{4\ell/p}(Q\times \bR,\alpha_2,\beta)$ is 1 when the holonomy representation of the flat connection $\beta$ is $\pm 1 \pmod p$.
\end{thm}
\begin{proof}
This follows from the proposition.
\end{proof}

\textbf{Proof of Theorem A:} Suppose there exists a smooth extension to $X_0$ with isolated fixed points. If a fixed-point of the $\pi$-action on $X_0$ has rotation numbers $(a,b)$ where $a,b$ are non-zero integers well-defined modulo $p$ then there is an equivariant lift corresponding to the $1$-parameter family of $\pi$-fixed ASD connections in $\cM_1^{\pi}(X,\theta)$ that it generates at the Taubes boundary. This is a $\widetilde{\pi}$-action on the principal $SU(2)$ bundle and has isotropy representation over the fixed point $p=(a,b)$ with weights $\pm (b-a)$ and the action on the $P\vv_{\End(X)}=\Sigma(2,3,5) \times [0,\infty) \times SU(2)$ is given by 
\begin{equation}
\widetilde{t} \cdot (x,s,U)=(tx,s,\phi(\widetilde{t})U)
\end{equation}
where $s\in \bR$, $U\in SU(2)$ and $\phi$ is the isotropy representation $\widetilde{\pi}\rightarrow SU(2)$ at $\infty$ with weights $\pm(a+b)$:
\begin{equation}
\widetilde{t} \mapsto \begin{pmatrix} \widetilde{t}^{a+b} &  \\  & \widetilde{t}^{-(a+b)} \end{pmatrix}
\end{equation}
we can mod out by the involution to get the $\pi$-equivariant adjoint $SO(3)$-bundle over $\Sigma(2,3,5) \times \bR$ with action given by the adjoint representation :$t \mapsto \begin{pmatrix} 1 &  \\  & t^{a+b} \end{pmatrix}$ with $\bZ/p=\langle t\rangle$. In the limit at $+\infty$ on $\Sigma(2,3,5) \times \bR$ the trivial product connection descends to a flat reducible connection on $Q$ whose $SO(3)$ holonomy representation is isomorphic to the adjoint isotropy representation $\ad \phi$. Since this holonomy is $\pm 1$ and $\pm 7 \pmod p$ this completes the proof.
\qed
\vspace{0.25cm}
\indent
The equivariant plumbing actions predict the existence of non-empty Floer type moduli spaces with fractional Yang-Mills energy, these dimensions can be computed by an index calculation using \cite{APS1}:
\begin{equation}
\dim \cM_{4\ell/p}(Q\times \bR,\alpha,\beta)=\frac{8\ell}{p}-\frac{1}{2}(h_{\alpha}+h_{\beta})+\frac{1}{2}(\rho_{\beta}(Q)-\rho_{\alpha}(Q)).
\end{equation}
Since $\alpha$ is irreducible and $\beta$ is reducible we have $h_{\alpha}=0$ and $h_{\beta}=1$. The rho invariants for reducible flat connections are determined by Kwasik-Lawson (\cite{KL93}, p.40) and are given by
\begin{multline}
\rho_{\beta}(Q)(l)=-\frac{2}{p}\sum_{k=1}^{p-1}\sin^{2}(\frac{\pi kl}{p})+\frac{2}{30p}\sum_{k=1}^{p-1}\csc^{2}(\frac{\pi k}{p})\sin^{2}(\frac{\pi kl}{p})\\
+\sum_{i=1}^{3}\frac{2}{pa_i}  \sum_{m_1=0}^{p-1} \sum_{m_2=1}^{a_{i-1}} \cot(\frac{\pi m_2}{a_i}) \cot(\frac{\pi m_1}{p}-\frac{\pi m_2 b_i}{a_i})\sin^2(\frac{\pi m_1 l}{p})
\end{multline}
where $l$ is the rotation number for the holonomy representation of $\beta$ in $SO(3)$. For irreducible flat connections $\alpha$, the rho invariants can be calculated by an $SO(3)$-flat cobordism to a union of lens spaces $L(a_i,pb_i)$ using the mapping cylinder for the Seifert fibration of $Q$ \cite{Yu91} as in the case of $\Sigma(a_1,a_2,a_3)$ \cite[p. 144]{Nikolai}. In this way, the linear equivariant plumbing actions give us that the moduli space $\cM_{\ell}(Q\times \bR,\alpha_2,\beta)$ for $\ell=1/120$ is non-empty and we get the following dimension:
\begin{equation}
\dim \cM_{\ell}(Q\times \bR,\alpha_2,\beta)=\dfrac{8}{p}(\dfrac{1}{120})-\frac{1}{2}+\frac{1}{2}(\rho_{\beta}(Q)(1)-\rho_{\alpha_2}(Q))=1.
\end{equation} 
If we now imagine turning-on a non-linear smooth $\pi$ extension to $X_0$, we don't know if $\cM_{\ell}(Q\times \bR,\alpha_1,\beta)$ for $\ell=49/120$ is non-empty but we have the following formal dimension:
\begin{equation}
\dim \cM_{\ell}(Q\times \bR,\alpha_1,\beta)=\dfrac{8}{p}(\dfrac{49}{120})-\frac{1}{2}+\frac{1}{2}(\rho_{\beta}(Q)(7)-\rho_{\alpha_1}(Q))=1.
\end{equation} 

We have obtained congruence relations that give constraints on the rotation data for the fixed points of a smooth extension. The next natural step is check these constraints against the $G$-signature theorem and we do this in the next two sections.

\section{G-Signature for $4$-Manifolds with Boundary}

For smooth, closed and orientable $4$-manifolds $X$, recall that the Hodge star operator induces an involution $\tau^2=1$ on the complexified sections of forms $\Omega^{\ast}=\oplus_k C^{\infty}(\Lambda^k TX\otimes \bC)$ splitting it into $\pm 1$ eigenspaces $\Omega^+ \oplus \Omega^-$. The signature operator $D^+=d+d^{\ast}$ restricted to $\Omega^+ \rightarrow \Omega^-$ is an elliptic operator whose index is the signature $\Sign(X)=b_2^+ -b_2^-$ of the non-degenerate quadratic form on $H^{2}(X;\bR)$. When a finite group $G$ acts by orientation preserving isometries on $X$, the cotangent bundle, as an equivariant bundle over $X$ has an action that commutes with the Hodge star operator, so we get a $G$-invariant elliptic operator $D^+$ whose $G$-index is a complex virtual representation of $G$.
\begin{equation}
\Ind_G(D^+)=H^+ - H^- \in R(G).
\end{equation}
The associated character or Lefschetz number $$L(g,D^+)=tr(g)(\Ind_G(D^+))=tr(g)\vv_{H^2_{+}}-tr(g)\vv_{H^2_{-}}$$ is called the $g$-signature and denoted by $\Sign(X,g)$. Note that when the action of $G$ is homologically trivial, the $g$-signature coincides with the usual signature. 

The $g$-signature can be computed from the fixed set by the Atiyah-Singer fixed point index theorem. We restrict to the case when $G$ is a finite cyclic group of odd prime order $\bZ/p=\langle t \rangle$, with $t=e^{2\pi i/p}$ and let $T_{p_i}X=\bC^2(a_i,b_i)$ be the local representation data over the fixed points $p_i$ for a homologically trivial action,  then 
\begin{equation}
\Sign(X_0)=\sum_i \left( \dfrac{t^{a_i}+1}{t^{a_i}-1} \right) \left( \dfrac{t^{b_i}+1}{t^{b_i}-1} \right)- 4\sum_{j} \dfrac{ \alpha_{j}t^{c_{j}}} {(t^{c_{j}}-1)^{2}} 
\end{equation}
where $\alpha_{j}$ are the self-intersections of the fixed $2$-spheres and $c_{j}$ is the rotation number of the normal bundle.

Let us now come to the situation where $X_0$ denotes a smooth, compact, simply connected four-manifold with boundary $\partial X_0=\Sigma$ an integral homology $3$-sphere. If a free $\bZ/p=\langle t\rangle$ action on $\Sigma$ extends to a locally linear, homologically-trivial action on $X_0$ (not necessarily free) then the $G$-signature theorem for manifolds with boundary is given in Atiyah-Patodi-Singer \cite{APS2}: 
\begin{equation}
\Sign(X_0,t)=L(X_0,t)-\eta_t(0)
\end{equation}
where $L(X_0,t)$ is the term occurring in the closed manifold case and  $\eta_t(0)$ is the equivariant eta invariant of $\Sigma$ or the G-signature defect. This invariant depends only on the $3$-manifold $\Sigma$ and not on how the action extends to the bounding four manifold $X_0$ nor does it depend on which four manifold is equivariantly bounding it. To see this, suppose the action on $\Sigma$ extends to another four manifold $X_1$, then consider the G-signature theorem on the closed four manifold $X_0 \cup_{\Sigma} -X_1$ to see that the $G$-signature defect terms are the same.

When the boundary $\partial X_0=\Sigma(a_1,a_2,a_3)$ is a Seifert fibered homology sphere, thought of as a link of a complex surface singularity, has a canonical negative definite resolution $\widetilde{X}_0$ given by plumbing disk bundles over $2$-spheres. Let $L(\widetilde{X}_0,t)$ denote the $g$-signature terms coming from the extension given by equivariant plumbing to an action on it's negative definite resolution $\widetilde{X}_0$. Since the actions are homologically trivial this gives the equation $L(X_0,t)-L(\widetilde{X}_0,t)=0$ in $\bQ(t)$. When written out using the contribution from the Lefschetz numbers we get:  
\begin{equation}\label{g-signature untwisted}
\sum_i \left( \dfrac{t^{a_i}+1}{t^{a_i}-1} \right) \left( \dfrac{t^{b_i}+1}{t^{b_i}-1} \right) - \sum_j \left( \dfrac{t^{\widetilde{a}_j}+1}{t^{\widetilde{a}_j}-1} \right) \left( \dfrac{t^{\widetilde{b}_j}+1}{t^{\widetilde{b}_j}-1} \right) + \sum_{\ell} \dfrac{4\widetilde{\alpha}_{\ell}t^{\widetilde{c}_{\ell}}}{(t^{\widetilde{c}_{\ell}}-1)^{2}}=0
\end{equation}
with rotation numbers $(\widetilde{a_j},\widetilde{b_j})$ and with $\widetilde{c_\ell}$ the rotation numbers on the normal bundle of the fixed-spheres in $\widetilde{X}_0$. We will follow \cite{HL89} to obtain congruence relations that relate the rotation numbers from a general extension with those of the linear plumbing action. To do this, recall that  
$\dfrac{t^a-1}{t-1}$
are units in the ring of cyclotomic integers when $a$ and $p$ are relatively prime. Multiplying both sides by $(t-1)^2$ induces an equation in the $R=\bZ[t]/(1+t+t^2+\cdots+t^{p-1})$. Let $I$ be the principal ideal generated by $(t-1)$, we will compute the lower order terms of the $I$-adic expansion
$
\widehat{R}=\varprojlim R/I^n
$
which has coefficients in $R/I\equiv \bZ/p$. To compute these coefficients we lift the equation to $\bZ[t]$ and compute the Taylor expansion about $t=1$ and reduce the coefficients modulo the prime $p$. The indeterminacy under the lift, which comes from the Taylor expansion of the cyclotomic polynomial $1+t+t^2\cdots t^{p-1}$, only affects terms of order greater than equal to $p-1$, since $p$ in the ring $R$ has $I$-adic valuation equal to $p-1$. The first term in equation \ref{g-signature untwisted} has expansion 
\begin{multline}
\sum_i   \dfrac{4}{a_ib_i} + \dfrac{4}{a_ib_i}(t-1) + \dfrac{1}{3} \left( \dfrac{a_i^2+b_i^2+1}{a_ib_i}\right) (t-1)^2 \\
-\dfrac{1}{180}\left(  \dfrac{a_i^4+b_i^4-5a_i^2b_i^2+3}{a_ib_i}  \right) (t-1)^4  \cdots 
\end{multline}
and similarly for the second expression. The last term has the Taylor expansion
\begin{multline}
 4 \sum_{\ell}   \dfrac{\widetilde{\alpha}_{\ell}}{\widetilde{c}_{\ell}^2} +  \dfrac{\widetilde{\alpha}_{\ell}}{\widetilde{c}_{\ell}^2}(t-1)^2  -\dfrac{1}{12}  \dfrac{\widetilde{\alpha}_{\ell}}{\widetilde{c}_{\ell}^2}(\widetilde{c}_{\ell}^2-1)(t-1)^2 \\
 +\dfrac{\widetilde{\alpha}_\ell}{120}\dfrac{(\widetilde{c}_\ell-1)(\widetilde{c}_\ell^3+\widetilde{c}_\ell^2+\widetilde{c}_\ell+1)}{\widetilde{c}_\ell^2} (t-1)^4 \cdots 
\end{multline}
which when equating coefficients give the following congruence relations:
\begin{subequations}
\begin{align}
& \sum_i \dfrac{4}{a_ib_i} - \sum_j \dfrac{4}{\widetilde{a}_j \widetilde{b}_j} + 4 \sum_\ell \dfrac{\widetilde{\alpha}_\ell}{\widetilde{c}_\ell^2} \equiv 0  \pmod p \\
&\dfrac{1}{3} \sum_i \dfrac{a_i^2+b_i^2+1}{a_ib_i}-\dfrac{1}{3} \sum_j \dfrac{\widetilde{a}_j^2+\widetilde{b}_j^2+1}{\widetilde{a}_j \widetilde{b}_j}+ \dfrac{1}{3} \sum_{\ell}\dfrac{\widetilde{\alpha}_\ell}{\widetilde{c}_\ell^2}(\widetilde{c}_\ell^2-1) \equiv 0 \pmod p\\
& \dfrac{-1}{180} \sum_i \dfrac{a_i^4+b_i^4-5a_i^2b_i^2+3}{a_ib_i} + \dfrac{1}{180} \sum_j \dfrac{\widetilde{a}_j^4+\widetilde{b}_j^4-5\widetilde{a}_j^2 \widetilde{b}_j^2+3}{\widetilde{a}_j \widetilde{b}_j}  \\ &+\sum_\ell \dfrac{4\widetilde{\alpha}_\ell}{120}\dfrac{(\widetilde{c}_\ell-1)(\widetilde{c}_\ell^3+\widetilde{c}_\ell^2+\widetilde{c}_\ell+1)}{\widetilde{c}_\ell^2} \equiv 0 \pmod p \notag
\end{align}
\end{subequations}
When $p \geq 7$, the values for $(\widetilde{a_j},\widetilde{b_j})$ which arise from equivariant plumbing along the $E_8$ diagram are given by 
$$ \{(-4,5),(-3,4),(-2,3),(-2,3),(-1,2),(-1,2),(-1,2)\}. $$ 
The central node in the plumbing diagram is a fixed $2$ sphere with self-intersection $\widetilde{\alpha}=-2$ and the rotation on the normal fiber is $\widetilde{c}\equiv 1 \pmod p$. Plugging in these values in the congruence equations proves the following
\begin{thm}
If $p\geq 7$, suppose a free $\bZ/p$ action on $\Sigma(2,3,5)$ extends to a locally linear, homologically trivial action on a four manifold $X_0$ with even negative definite intersection form $Q_{X_0}=-E_8$ with only isolated fixed points with local rotation data $\bC^2(a_i,b_i)$, then the following congruence relations for the rotation numbers hold:
\begin{align}
& \sum_i \dfrac{1}{a_ib_i} \equiv \dfrac{1}{30}  \pmod p\\
& \sum_i \dfrac{a_i^2+b_i^2+1}{a_ib_i} \equiv \dfrac{-269}{15} \pmod p \\
& \sum_i  \dfrac{a_i^4+b_i^4-5a_i^2b_i^2+3}{a_ib_i} \equiv \dfrac{1712}{15} \pmod p
\end{align}
\end{thm}
\begin{ex}
The expression for the equivariant eta invariant can be determined by equivariant plumbing as above. For $p\geq 7$ it is given by 
$$ 
\eta_t(0)=\left( \dfrac{t^{\widetilde{a}_i}+1}{t^{\widetilde{a}_i}-1} \right) \left( \dfrac{t^{\widetilde{b}_i}+1}{t^{\widetilde{b}_i}-1}\right)+\dfrac{8t}{(t-1)^2}+8.
$$
When $p=7$ it can be checked that the following rotation data $\{(a_i,b_i)\}=\{(1,1) \times 3, (1,-3),(1,-1) \times 2, (2,2) \times 2,(3,3)\}$ solve the $G$-signature theorem:
$$
-8=\sum_i \left( \dfrac{t^{a_i}+1}{t^{a_i}-1} \right) \left( \dfrac{t^{b_i}+1}{t^{b_i}-1} \right)-\eta_t(0)
$$
and satisfy 
\begin{gather*}
\sum_i \dfrac{1}{a_ib_i} \equiv 4 \equiv \dfrac{1}{30} \pmod 7\\
\sum_i \dfrac{a_i^2+b_i^2+1}{a_ib_i} \equiv 4 \equiv \dfrac{-269}{15} \pmod 7\\
\sum_i  \dfrac{a_i^4+b_i^4-5a_i^2b_i^2+3}{a_ib_i}\equiv 4 \equiv \dfrac{1712}{15} \pmod 7.
\end{gather*}
\end{ex}

\section{Twisted G-Signature Relations}

Consider now twisting the $G$-signature operator $D^+$ with an equivariant vector bundle $E \in K_G(X)$ of complex rank $2$ over $X$:
\begin{equation}
D^+_E: \Lambda^+\otimes E \rightarrow \Lambda^- \otimes E .
\end{equation}
The index is given by the following formula, see Donnelly \cite[p.901]{Don78}:
\begin{equation}
\chi(g):=\Sign(X,g,E)=\sum_{N \in \Fix(g)} \int_N ch_g(E \vv_N)\cL(N) -\eta_g(0,E)
\end{equation}
where the equivariant Chern character is well-defined over the fixed set since $K_G(X^g)=K(X^g)\otimes R(G)$ \cite{segal68}. Our goal is to extract one more congruence relation as we did for the untwisted $G$-signature operator. This time the congruence will relate not only the tangential representations but also the isotropy representations on the equivariant bundle $E$ over the fixed set.
Consider two distinct equivariant lifts which agree on the $\End(X)=\Sigma(2,3,5) \times [0,\infty)$; the topological index gives two virtual representations and taking the difference of their associated characters gives:
\begin{equation}
\chi(t)-\widetilde{\chi}(t)=L(X,t,E)-L(\widetilde{X},t,E).
\end{equation}
Since $\chi(t)$ and $\widetilde{\chi}(t)$ are virtual characters and $\chi(1)=\widetilde{\chi}(1)$ we can carry out the same procedure to get congruence relations that relate one equivariant lift to another. First we will compute the contribution of the fixed points as in \cite {Shanahan78}:
 
\begin{equation}
\begin{split}
L(g,D^+_E)\vv_{p_i}&=L_{\theta}(N^g_{\theta})\ch_g(E)[p_i]\\
&=\dfrac{\ch_g(\Lambda^+-\Lambda^-)(N^g \otimes \bC)\ch_g(E)}{\ch_g(\Lambda_{-1}N^g\otimes \bC)}[p_i]\\           
          &=\dfrac{(e^{-2\pi a_i i/p}-e^{2\pi a_i i/p})(e^{-2\pi b_i i/p}-e^{2\pi b_i i/p})(e^{2\pi i \lambda/p}+e^{-2\pi i \lambda/p})}{(1-e^{2\pi a_i i/p})(1-e^{-2\pi a_i i/p})(1-e^{2\pi b_i i/p})(1-e^{-2\pi b_i i/p})} \\
          &=\coth(a_i \pi i/p)\coth(b_i \pi i/p)(e^{2\pi i \lambda/p}+e^{-2\pi i \lambda/p}) \\
          &=i^2\cot(a_i \pi/p)\cot(b_i \pi/p)(e^{2\pi i \lambda/p}+e^{-2\pi i \lambda/p}) \\
          &= \left( \dfrac{t^{a_i}+1}{t^{a_i}-1} \right) \left( \dfrac{t^{b_i}+1}{t^{b_i}-1} \right) (t^{\lambda}+t^{-\lambda})
\end{split}
\end{equation}
where as a $\pi$-equivariant $SU(2)$ bundle, $E=\bC(\lambda) \oplus \bC(-\lambda)$ and $\ch_g(E)=t^\lambda+t^{-\lambda}$. Similarly, we can compute the contribution from a fixed $2$-sphere $F$:
\begin{equation}
\begin{split}
L(g,D^+_E)\vv_F&=L(T^g)L_{\theta}(N^g_{\theta})\ch_g(E)[F]\\
&=x\coth(x/2)\coth(\frac{y+i\theta}{2})\{t^{\lambda}+t^{-\lambda}\}[F] \\
          &=\{2+\dfrac{x^2}{6}+\cdots\}\{\coth(i\theta/2)-\frac{1}{2} \csch ^2(i\theta/2)y \cdots\} \{t^{\lambda}+t^{-\lambda}\}[F] \\           
          &=-\csch^2(i\theta/2)y (t^{\lambda}+t^{-\lambda})[F]\\
          &=\dfrac{-4t^c}{(t^c-1)^2}F \cdot F (t^{\lambda}+t^{-\lambda})
\end{split}
\end{equation}
where $x$ is the Euler class of $F$ and $y$ the Euler class of the normal bundle. Now the Taylor expansion on the right hand side of $$(\chi(t)-\widetilde{\chi}(t))(t-1)^2=(L(X,t,E)-\widetilde{L}(X,t,E))(t-1)^2$$ up to order two has the form:
\begin{align*}
\sum_i \dfrac{8}{a_ib_i}+\dfrac{8}{a_ib_i}(t-1) +\dfrac{2}{3}\left( \dfrac{6\lambda_i^2+a_i^2+b_i^2+1}{a_ib_i}\right)(t-1)^2\\
   - \bigg( \sum_j \dfrac{8}{\widetilde{a_j}\widetilde{b_j}}+ \dfrac{8}{\widetilde{a_j}\widetilde{b_j}}(t-1)+\dfrac{2}{3}\left( \dfrac{6\widetilde{\lambda_j}^2+\widetilde{a_j}^2+\widetilde{b_j}^2+1}{\widetilde{a_j}\widetilde{b_j}}\right)(t-1)^2\\
  +\dfrac{16}{\widetilde{c}^2}+\dfrac{16}{\widetilde{c}^2}(t-1)-\dfrac{4}{3}\left(\dfrac{\widetilde{c}^2-1-6\widetilde{\lambda}^2}{c^2}\right)(t-1)^2 \bigg)
\end{align*}
and the second order term of the expansion of the left hand side vanishes. Here $(a_i,b_i)$ are rotation numbers for an extension with just isolated fixed points with $\lambda_i$ the isotropy representations in $E$. The $(\widetilde{a_j},\widetilde{b_j})$ are rotation numbers from equivariant plumbing extension with $\widetilde{\lambda}_j$ the isotropy over the fixed points and $\widetilde{c}$ the rotation number on the normal fiber over the fixed $2$-sphere and isotropy $\widetilde{\lambda}$ for the equivariant lift over this sphere.

Plugging in the values from equivariant plumbing and equating the second order coefficients gives the following relation: 
\begin{equation}
4\sum_i \dfrac{\lambda_i^2}{a_ib_i}-4\sum_j \dfrac{\widetilde{\lambda_j}^2}{\widetilde{a_j}\widetilde{b_j}}+\dfrac{4}{3}\left( \dfrac{\widetilde{c}^2-1-6\widetilde{\lambda^2}}{\widetilde{c}^2} \right)\equiv 0 \pmod p
\end{equation} 
Since $\widetilde{c}=1$ we have 
\begin{equation}\label{twisted relation}
4\sum_i \dfrac{\lambda_i^2}{a_ib_i}\equiv 4\sum_j \dfrac{\widetilde{\lambda_j}^2}{\widetilde{a_j}\widetilde{b_j}}+8\widetilde{\lambda}^2 \pmod p.
\end{equation}   
To summarize, this congruence relation relates the tangential representations at fixed point set as well as the isotorpy representations over the fibers for two different extensions and is valid so long as the the choices of equivariant lifts agree on $\Sigma(2,3,5)$ over the $\End(X)$.
\noindent
\textbf{Proof of Theorem B:}
\noindent
The previous untwisted congruence relation
\begin{equation}
\sum_i \dfrac{a_i^2+b_i^2+1}{a_ib_i} \equiv \dfrac{-269}{15}\pmod p
\end{equation}
gives the following when we add and subtract $2a_ib_i$ from the numerator:
\begin{equation}\label{sum with squares}
\sum_{i=1}^9 \dfrac{(a_i+b_i)^2}{a_ib_i}\equiv \dfrac{1}{30} \pmod p
\end{equation}
By Theorem A we have the possibilities $a_i+b_i \equiv \pm 1$ or $a_i+b_i\equiv \pm 7 \pmod p$. We can then split this sum as $U+49V\equiv 1/30 \pmod p$ where $U=\sum 1/a_ib_i$ with the sum over points that satisfy  $a_i+b_i \equiv \pm 1$ and similarly for $V$. But the relation $\sum_i 1/a_ib_i \equiv 1/30$ can also be split as $U+V\equiv 1/30 \pmod p$. Together they show at once that $V\equiv 0 \pmod p$ and that not all points can satisfy $a_i+b_i\equiv \pm 7 \pmod p$. 

So there is at least one point $(a,b)$ with $a+b \equiv \pm 1 \pmod p$ and this agrees with the equivariant lift on the $\End(X)$ of the plumbing action from the $3$-dimensional stratum in the moduli space over the negative definite resolution $\widetilde{X}_0$. The isotropy representations for this equivariant lift has $\widetilde{\lambda}=\widetilde{\lambda_j}=\pm 1/2 \pmod p$. Plugging these values into equation \ref{twisted relation} gives the relation
\begin{equation}
4\sum_i \dfrac{\lambda_i^2}{a_ib_i} \equiv \dfrac{1}{30}.
\end{equation}
Now choosing the equivariant lift arising from the fixed point $(a,b)$ gives isotropy $(b-a)/2$ over $(a,b)$ and $(a+b)/2 \equiv \pm 1/2$ over all the other fixed points:  
\begin{equation}
\dfrac{(b-a)^2}{ab} +\sum_{i=1}^8 \dfrac{1}{a_ib_i}\equiv \dfrac{1}{30} \pmod p.
\end{equation}
Similarly we can write equation \ref{sum with squares} as
\begin{equation}
\dfrac{(a+b)^2}{ab}+\sum_{i=1}^8 \dfrac{(a_i+b_i)^2}{a_ib_i} \equiv \dfrac{1}{30} \pmod p
\end{equation}
and subtracting the two equations, noting that there is no contribution from $V$ gives $-4\equiv 0 \pmod p$ which contradicts our prime $p$ is odd.\qed

\vspace{.5 cm}
It is worth noting that a similar analysis can be carried out with other Brieskorn homology $3$-spheres. It is of particular interest to study $\bZ/p$ actions on homotopy $K3$ surfaces containing an equivariant copy of $\Sigma(2,7,13)$.

\bibliographystyle{alpha}
\bibliography{Bibliography}

\begin{thebibliography}{MMR94}

\bibitem[APS75a]{APS1}
M.~F. Atiyah, V.~K. Patodi, and I.~M. Singer.
\newblock Spectral asymmetry and {R}iemannian geometry. {I}.
\newblock {\em Math. Proc. Cambridge Philos. Soc.}, 77:43--69, 1975.

\bibitem[APS75b]{APS2}
M.~F. Atiyah, V.~K. Patodi, and I.~M. Singer.
\newblock Spectral asymmetry and {R}iemannian geometry. {II}.
\newblock {\em Math. Proc. Cambridge Philos. Soc.}, 78(3):405--432, 1975.

\bibitem[Auc94]{Auckly94b}
David~R. Auckly.
\newblock Topological methods to compute {C}hern-{S}imons invariants.
\newblock {\em Math. Proc. Cambridge Philos. Soc.}, 115(2):229--251, 1994.

\bibitem[Bie77]{Bie77}
Edward Bierstone.
\newblock General position of equivariant maps.
\newblock {\em Trans. Amer. Math. Soc.}, 234(2):447--466, 1977.

\bibitem[BKS90]{BKS90}
N.P. Buchdahl, S{\l}awomir Kwasik, and Reinhard Schultz.
\newblock One fixed point actions on low-dimensional spheres.
\newblock {\em Inventiones mathematicae}, 102(1):633--662, 1990.

\bibitem[BM93]{BM93}
Peter~J. Braam and Gordana Mati{\'c}.
\newblock The {S}mith conjecture in dimension four and equivariant gauge
  theory.
\newblock {\em Forum Math.}, 5(3):299--311, 1993.

\bibitem[Cho91]{cho91}
Yong~Seung Cho.
\newblock Finite group actions on the moduli space of self-dual connections.
  {I}.
\newblock {\em Trans. Amer. Math. Soc.}, 323(1):233--261, 1991.

\bibitem[CK08]{CK2008}
Weimin Chen and Slawomir Kwasik.
\newblock Symmetries and exotic smooth structures on a k3 surface.
\newblock {\em Journal of Topology}, 1(4):923--962, 2008.

\bibitem[CR62]{CR62}
Charles~W Curtis and Irving Reiner.
\newblock {\em Representation theory of finite groups and associative
  algebras}, volume 356.
\newblock AMS Bookstore, 1962.

\bibitem[DK90]{DK90}
S.~K. Donaldson and P.~B. Kronheimer.
\newblock {\em The geometry of four-manifolds}.
\newblock Oxford Mathematical Monographs. The Clarendon Press Oxford University
  Press, New York, 1990.
\newblock Oxford Science Publications.

\bibitem[Don78]{Don78}
Harold Donnelly.
\newblock Eta invariants for {$G$}-spaces.
\newblock {\em Indiana Univ. Math. J.}, 27(6):889--918, 1978.

\bibitem[Don83]{D83}
S.~K. Donaldson.
\newblock An application of gauge theory to four-dimensional topology.
\newblock {\em J. Differential Geom.}, 18(2):279--315, 1983.

\bibitem[Don87]{Don87}
S.~K. Donaldson.
\newblock The orientation of {Y}ang-{M}ills moduli spaces and {$4$}-manifold
  topology.
\newblock {\em J. Differential Geom.}, 26(3):397--428, 1987.

\bibitem[Don02]{D02}
S.~K. Donaldson.
\newblock {\em Floer homology groups in {Y}ang-{M}ills theory}, volume 147 of
  {\em Cambridge Tracts in Mathematics}.
\newblock Cambridge University Press, Cambridge, 2002.
\newblock With the assistance of M. Furuta and D. Kotschick.

\bibitem[Edm87]{E87}
Allan~L. Edmonds.
\newblock Construction of group actions on four-manifolds.
\newblock {\em Trans. Amer. Math. Soc.}, 299(1):155--170, 1987.

\bibitem[Edm89]{edmonds1989aspects}
Allan~L Edmonds.
\newblock Aspects of group actions on four-manifolds.
\newblock {\em Topology and its Applications}, 31(2):109--124, 1989.

\bibitem[FL86]{FL86}
Ronald Fintushel and Terry Lawson.
\newblock Compactness of moduli spaces for orbifold instantons.
\newblock {\em Topology Appl.}, 23(3):305--312, 1986.

\bibitem[Flo88]{Floer88}
Andreas Floer.
\newblock An instanton-invariant for {$3$}-manifolds.
\newblock {\em Comm. Math. Phys.}, 118(2):215--240, 1988.

\bibitem[Fr{\o}96]{Froyshov96}
Kim~A. Fr{\o}yshov.
\newblock The {S}eiberg-{W}itten equations and four-manifolds with boundary.
\newblock {\em Math. Res. Lett.}, 3(3):373--390, 1996.

\bibitem[FS85]{FS85}
Ronald Fintushel and Ronald~J. Stern.
\newblock Pseudofree orbifolds.
\newblock {\em Ann. of Math. (2)}, 122(2):335--364, 1985.

\bibitem[FS90]{FS90}
Ronald Fintushel and Ronald~J. Stern.
\newblock Instanton homology of {S}eifert fibred homology three spheres.
\newblock {\em Proc. London Math. Soc. (3)}, 61(1):109--137, 1990.

\bibitem[FS91]{FS91}
Ronald Fintushel and Ronald Stern.
\newblock Homotopy {$K3$} surfaces containing {$\Sigma$ (2, 3, 7)}.
\newblock {\em J. Differential Geom}, 34:255--265, 1991.

\bibitem[FS94]{FS94}
Ronald Fintushel and Ronald~J. Stern.
\newblock Surgery in cusp neighborhoods and the geography of irreducible
  {$4$}-manifolds.
\newblock {\em Invent. Math.}, 117(3):455--523, 1994.

\bibitem[FT77]{FT77}
Michael~H Freedman and Lawrence Taylor.
\newblock $\lambda$-splitting $4$-manifolds.
\newblock {\em Topology}, 16:181--184, 1977.

\bibitem[FU91]{FU91}
Daniel~S. Freed and Karen~K. Uhlenbeck.
\newblock {\em Instantons and four-manifolds}, volume~1 of {\em Mathematical
  Sciences Research Institute Publications}.
\newblock Springer-Verlag, New York, second edition, 1991.

\bibitem[Fur89]{F89}
Mikio Furuta.
\newblock A remark on a fixed point of finite group action on {$S^4$}.
\newblock {\em Topology}, 28(1):35--38, 1989.

\bibitem[GM93]{GM93}
Robert~E. Gompf and Tomasz~S. Mrowka.
\newblock Irreducible {$4$}-manifolds need not be complex.
\newblock {\em Ann. of Math. (2)}, 138(1):61--111, 1993.

\bibitem[HL92]{HL92}
Ian Hambleton and Ronnie Lee.
\newblock Perturbation of equivariant moduli spaces.
\newblock {\em Math. Ann.}, 293(1):17--37, 1992.

\bibitem[HL95]{HL95}
Ian Hambleton and Ronnie Lee.
\newblock Smooth group actions on definite {$4$}-manifolds and moduli spaces.
\newblock {\em Duke Math. J.}, 78(3):715--732, 1995.

\bibitem[HLM89]{HL89}
Ian Hambleton, Ronnie Lee, and Ib~Madsen.
\newblock Rigidity of certain finite group actions on the complex projective
  plane.
\newblock {\em Comment. Math. Helv.}, 64(4):618--638, 1989.

\bibitem[HR78]{HR78}
I.~Hambleton and C.~Riehm.
\newblock Splitting of {H}ermitian forms over group rings.
\newblock {\em Invent. Math.}, 45(1):19--33, 1978.

\bibitem[HT04]{HT04}
Ian Hambleton and Mihail Tanase.
\newblock Permutations, isotropy and smooth cyclic group actions on definite
  4-manifolds.
\newblock {\em Geom. Topol.}, 8:475--509, 2004.

\bibitem[KBL88]{KL88}
SŁawomir Kwasik and Kyung Bai~Lee.
\newblock Locally linear actions on 3-manifolds.
\newblock {\em Mathematical Proceedings of the Cambridge Philosophical
  Society}, 104:253--260, 9 1988.

\bibitem[Kiy11]{kiyono2011}
Kazuhiko Kiyono.
\newblock Nonsmoothable group actions on spin 4--manifolds.
\newblock {\em Algebraic \& Geometric Topology}, 11:1345--1359, 2011.

\bibitem[KL93]{KL93}
S{\l}awomir Kwasik and Terry Lawson.
\newblock Nonsmoothable {$Z_p$} actions on contractible {$4$}-manifolds.
\newblock {\em J. Reine Angew. Math.}, 437:29--54, 1993.

\bibitem[Law85]{Lawson85}
H.~Blaine Lawson, Jr.
\newblock {\em The theory of gauge fields in four dimensions}, volume~58 of
  {\em CBMS Regional Conference Series in Mathematics}.
\newblock Published for the Conference Board of the Mathematical Sciences,
  Washington, DC, 1985.

\bibitem[LN08]{NL2008}
Ximin Liu and Nobuhiro Nakamura.
\newblock Nonsmoothable group actions on elliptic surfaces.
\newblock {\em Topology and its Applications}, 155(9):946 -- 964, 2008.

\bibitem[LS92]{LS92}
E.~Luft and D.~Sjerve.
\newblock On regular coverings of {$3$}-manifolds by homology {$3$}-spheres.
\newblock {\em Pacific J. Math.}, 152(1):151--163, 1992.

\bibitem[MMR94]{MMR}
John~W. Morgan, Tomasz Mrowka, and Daniel Ruberman.
\newblock {\em The $L^2$-Moduli Space and a Vanishing Theorem for Donaldson
  Polynomial Invariants}.
\newblock International Press, 1994.

\bibitem[Que81]{quebbemann1981klassifikation}
Heinz-Georg Quebbemann.
\newblock Zur klassifikation unimodularer gitter mit isometrie von
  primzahlordnung.
\newblock {\em Journal f{\"u}r die reine und angewandte Mathematik},
  326:158--170, 1981.

\bibitem[Sav00]{Nikolai}
Nikolai Saveliev.
\newblock {\em Invariants for Homology 3-Spheres}.
\newblock Springer, 2000.

\bibitem[Seg68]{segal68}
Graeme Segal.
\newblock Equivariant k-theory.
\newblock {\em Publications Math{\'e}matiques de l'IH{\'E}S}, 34(1):129--151,
  1968.

\bibitem[Sha78]{Shanahan78}
Patrick Shanahan.
\newblock {\em The {A}tiyah-{S}inger index theorem}, volume 638 of {\em Lecture
  Notes in Mathematics}.
\newblock Springer, Berlin, 1978.
\newblock An introduction.

\bibitem[Tau82]{taubes82}
Clifford~Henry Taubes.
\newblock Self-dual yang-mills connections on non-self-dual
  $\{$4$\}$-manifolds.
\newblock {\em Journal of Differential Geometry}, 17(1):139--170, 1982.

\bibitem[Tau87]{T87}
Clifford~Henry Taubes.
\newblock Gauge theory on asymptotically periodic {$4$}-manifolds.
\newblock {\em J. Differential Geom.}, 25(3):363--430, 1987.

\bibitem[Tau93]{taubes1993}
Clifford Taubes.
\newblock {\em $L^2$ Moduli Spaces on 4-manifolds with Cylindrical Ends}.
\newblock International Press, 1993.

\bibitem[Yu91]{Yu91}
Bao~Zhen Yu.
\newblock A note on an invariant of {F}intushel and {S}tern.
\newblock {\em Topology Appl.}, 38(2):137--145, 1991.

\end{thebibliography}

\end{document}